\documentclass[11pt]{amsart}
\usepackage{amssymb, amsmath, latexsym,   bm, mathrsfs}
\usepackage{multirow, xcolor, float}

\usepackage{longtable}

\makeatletter
\def\@cline#1-#2\@nil{%
  \omit
  \@multicnt#1%
  \advance\@multispan\m@ne
  \ifnum\@multicnt=\@ne\@firstofone{&\omit}\fi
  \@multicnt#2%
  \advance\@multicnt-#1%
  \advance\@multispan\@ne
  \leaders\hrule\@height\arrayrulewidth\hfill
  \cr
  \noalign{\nobreak\vskip-\arrayrulewidth}}
\makeatother

\renewcommand{\baselinestretch}{\baselinestretch}
\renewcommand{\baselinestretch}{1.1}
\numberwithin{equation}{section}

\newtheorem{thm}{Theorem}[section]
\newtheorem{lem}[thm]{Lemma}

\newtheorem{defn}[thm]{Definition}
\newtheorem{rmk}[thm]{Remark}

\newcommand{\ord}{\textnormal{ord}}
\newcommand{\diag}{\text{diag}}
\newcommand{\z}{{\mathbb Z}}
\newcommand{\q}{{\mathbb Q}}

\newcommand{\0}{\bm 0}

\newcommand{\fs}{\mathfrak s}

\newcommand{\Mod}[1]{\ (\mathrm{mod}\ #1)}

\newcommand{\fn}{\mathfrak n}
\newcommand{\ang}[1]{\langle #1 \rangle}

\begin{document}

\title{Primitively $2$-universal senary integral quadratic forms}

\author{Byeong-Kweon Oh and Jongheun Yoon}

\address{Department of Mathematical Sciences and Research Institute of Mathematics, Seoul National University, Seoul 08826, Korea}
\email{bkoh@snu.ac.kr}
\thanks{This work of the second author was supported by the National Research Foundation of Korea(NRF) grant funded by the Korea government(MSIT) (NRF-2019R1A2C1086347 and NRF-2020R1A5A1016126).}

\address{Department of Mathematical Sciences, Seoul National University, Seoul 08826, Korea}
\email{jongheun.yoon@snu.ac.kr}



\begin{abstract}
For a positive integer $m$, a (positive definite integral) quadratic form is called primitively $m$-universal if it primitively represents all quadratic forms of rank $m$. It was proved in \cite{JKKKO23} that there are exactly $107$ equivalence classes of primitively $1$-universal quaternary quadratic forms. In this article, we prove that the minimal rank of primitively $2$-universal quadratic forms is six, and there are exactly $201$ equivalence classes of primitively $2$-universal senary quadratic forms.
\end{abstract}

\maketitle

\section{Introduction}
\label{sec:intro}

A quadratic form of rank $n$ is a quadratic homogeneous polynomial
\[
 f(x_1, \dots, x_n) = \sum_{i, j = 1}^n f_{ij}x_i x_j \qquad (f_{ij} = f_{ji}\in\q)\text{,}
\]
where the corresponding symmetric matrix $M_f = (f_{ij})$ is nondegenerate. We say $f$ is integral if $M_f$ is an integral matrix, and we say $f$ is positive definite if $M_f$ is positive definite. Throughout this article, we always assume that any quadratic form is integral and positive definite.

For two (positive definite integral) quadratic forms $f$ and $g$ of rank $n$ and $m$, respectively, we say $g$ is represented by $f$ if there is an integral matrix $T\in M_{n, m}(\z)$ such that $T^t M_f T = M_g$. We say $g$ is isometric to $f$ if the above integral matrix $T$ is invertible. The equivalence class of $f$ is the set of all quadratic forms which are isometric to $f$. We say $g$ is primitively represented by $f$ if the above integral matrix $T$ can be extended to an invertible matrix in $GL_n(\z)$ by adding suitable $(n-m)$ columns. In particular, a positive integer $a$ is primitively represented by $f$ if and only if there are integers $x_1$, \dots, $x_n$ such that
\[
 f(x_1, \dots, x_n) = a \quad\text{and}\quad \gcd(x_1, \dots, x_n) = 1\text{.}
\]

For a positive integer $m$, a quadratic form is called (primitively) $m$-universal if it (primitively, respectively) represents all quadratic forms of rank $m$. Lagrange's four-square theorem states that the quaternary quadratic form corresponding to the identity matrix $I_4$ is $1$-universal. The complete classification of $1$-universal quadratic forms up to isometry has been done by Ramanujan, Dickson, Conway--Schneeberger, and Bhargava (see \cite{Ra17}, \cite{Di27}, and \cite{Bh99}). In 1998, Kim, Kim, and the first author proved in \cite{KKO98} that there are exactly eleven equivalence classes of $2$-universal quinary quadratic forms. For more information on $n$-universal quadratic forms, one may see \cite{Kim04} and \cite{Oh00}.

Let $f$ and $g$ be quadratic forms of rank $n$ and $m$, respectively. We say $g$ is represented by $f$ over the ring of $p$-adic integers $\z_p$ for some prime $p$, if there is a matrix $T\in M_{n, m}(\z_p)$ such that $T^t M_f T = M_g$. We further say $g$ is primitively represented by $f$ over $\z_p$ if the above matrix $T$ can be extended to an invertible matrix in $GL_n(\z_p)$ by adding suitable $(n-m)$ columns in $\z_p^n$. Clearly, if $g$ is (primitively) represented by $f$ over $\z$, then $g$ is also (primitively) represented by $f$ over $\z_p$ for any prime $p$. However, the converse is not true in general. In fact, there is an effective criterion whether or not $g$ is represented by $f$ over $\z_p$ for any prime $p$ (for this, see \cite{OM58}). However, as far as the authors know, there is no known effective criterion whether or not $g$ is primitively represented by $f$ over $\z_p$.

Finding primitively $1$-universal quadratic forms was first considered by Budarina in \cite{Bu10}. She classified all primitively $1$-universal quaternary quadratic forms satisfying some special local properties. Later, she also classified in \cite{Bu11} all primitively $2$-universal quadratic forms which has squarefree odd discriminant and of class number one. Recently, Earnest and Gunawardana classified in \cite{EG21} all quadratic forms that primitively represent all integers corresponding to unary quadratic forms over $\z_p$ for any prime $p$ including $p=2$. Furthermore, they provided a complete list of all $1$-universal quaternary quadratic forms that are almost primitively $1$-universal. Here, a quadratic form is called almost primitively $1$-universal if it represents almost all positive integers primitively. Ju, Kim, Kim, Kim, and the first author finally proved in \cite{JKKKO23} that there are exactly $107$ equivalence classes of primitively $1$-universal quaternary quadratic forms. In this article, we prove that the minimal rank of primitively $2$-universal quadratic forms is six, and that there are exactly $201$ equivalence classes of primitively $2$-universal senary quadratic forms (see Table~\ref{tblr:201}).

The subsequent discussion will be conducted in the more suitable language of quadratic spaces and lattices. Let $F$ be either $\q$ or $\q_p$ for some prime $p$. An (quadratic) $F$-space is a finite dimensional vector space $V$ equipped with a nondegenerate symmetric bilinear form
\[
 B: V\times V\to F
\]
and the associated quadratic form $Q: V\to F$ is defined by $Q(v) = B(v, v)$ for any $v\in V$. Let $R$ be either $\z$ or $\z_p$ for some prime $p$. An (quadratic) $R$-lattice is a finitely generated free $R$-module $L$ equipped with a nondegenerate symmetric bilinear form
\[
 B: L\times L\to R
\]
and an associated quadratic form $Q: L\to R$ is defined by $Q(v) = B(v, v)$ for any $v\in L$. We will denote by $Q(L)$ the set of $Q(v)$ for any vector $v\in L$. The scale $\fs L$ of $L$ is the ideal $B(L, L)$ in $R$, and the norm $\fn L$ of $L$ is the ideal in $R$ generated by the set $Q(L)$.

Given an $R$-basis $e_1, \dotsc, e_n$ for an $R$-lattice $L$, we call the corresponding symmetric $n\times n$ matrix $M_L = (B(e_i, e_j))$ the Gram matrix of $L$, and in this case, we write $L\cong (B(e_i, e_j))$. We call a $\z$-lattice $L$ positive definite if $M_L$ is positive definite. For a symmetric $n\times n$ matrix $N$ over $R$, we let $\langle N\rangle$ (or simply $N$) stand for any $\z$-lattice whose Gram matrix is $N$. For instance, the $\z$-lattice $I_n$ stands for the $n$-ary lattice whose Gram matrix is the identity matrix of rank $n$. Moreover, if the Gram matrix of $L$ is a diagonal matrix $\diag(a_1, \dotsc, a_n)$ ($a_i\in R$), then we simply write $L\cong \langle a_1, \dotsc, a_n\rangle$. For a $\z$-lattice $L$ and a positive integer $a$, $L^a$ denotes the $\z$-lattice obtained from scaling $L$ by $a$ so that 
\[
 M_{L^a}=a\cdot M_L\text{.}
\]
For a $\z$-lattice $L$, we define $L_p = \z_p\otimes L$, which is a $\z_p$-lattice for any prime $p$. From now on, we always assume that any $\z$-lattice $L$ is integral, that is, $\fs(L)\subseteq\z$, and positive definite, that is, $M_L$ is positive definite.

Let $m$ and $n$ be positive integers such that $m\le n$. An $m\times n$ ($n\times m$) matrix $A$ over $R$ is called primitive if it can be extended to an invertible matrix in $GL_n(R)$ by adding suitable $n-m$ rows (columns, respectively), or equivalently, if the greatest common divisor of the determinants of all $m\times m$ submatrices of $A$ is a unit in $R$. A finite sequence $v_1, \dots, v_m$ of vectors in $R^n$ is called primitive if it can be extended to a basis for $R^n$, or equivalently, if the coefficient matrix of $v_1, \dots, v_m$ is primitive. Finally, a submodule $M$ of $R^n$ is primitive if and only if there exists a basis for $M$ which is primitive, or equivalently, $M$ is a direct summand of $R^n$. Note that a vector $v = (a_1, \dots, a_n)\in R^n$ is primitive if and only if
\[
 \gcd(a_1, \dots, a_n) \in R^\times\text{.}
\]
For an $R$-lattice $L$, we will denote by $Q^\ast(L)$ the set of $Q(v)$ for any primitive vector $v\in L$.

\newdimen\ofontdimentwo%
\newdimen\ofontdimensam%
\newdimen\ofontdimennes%
\newdimen\ofontdimenqil%
\ofontdimentwo=\fontdimen2\font%
\ofontdimensam=\fontdimen3\font%
\ofontdimennes=\fontdimen4\font%
\ofontdimenqil=\fontdimen7\font%
\fontdimen2\font=3.3pt%
\fontdimen3\font=1.65pt%
\fontdimen4\font=1.1pt%
\fontdimen7\font=1.1pt%
For two $R$-lattices $M$ and $L$, a representation from $M$ to $L$ is a linear map $\sigma : M \to L$ such that $B(\sigma(x), \sigma(y)) = B(x, y)$ for any $x$, $y\in L$. If in addition $\sigma(M)$ is a primitive sublattice of $L$, then such a representation is called primitive. If there exists a (primitive) representation $\sigma : M \to L$ then we say that $M$ is (primitively, respectively) represented by $L$. If such a representation $\sigma$ is bijective then we say that $M$ is isometric to $L$, and in that case, we denote it by $\sigma: M\cong L$. An $R$-lattice $L$ is called (primitively) $n$-universal if $L$ (primitively, respectively) represents all $R$-lattices of rank $n$.
\fontdimen2\font=\ofontdimentwo%
\fontdimen3\font=\ofontdimensam%
\fontdimen4\font=\ofontdimennes%
\fontdimen7\font=\ofontdimenqil%

For an $R$-lattice $L$, the class of $L$ consists of all $R$-lattices that are isometric to $L$. Let $L$ and $M$ be $\z$-lattices. If $M_p$ is isometric to $L_p$ for any prime $p$, then we say that $M$ is in the genus of $L$. It is well-known that genus is partitioned into finitely many classes. It is well-known that if $M_p$ is (primitively) represented by $L_p$ for any prime $p$, then $M$ is (primitively, respectively) represented by some $\z$-lattice $L'$ in the genus of $L$ (see \cite[102:5]{OM}).

We always assume that $\mathbb{H}$ and $\mathbb{A}$ are binary $R$-lattices whose Gram matrices are
\[
 \mathbb{H}\cong \left(\begin{smallmatrix}0&1\\1&0\end{smallmatrix}\right) \quad\text{and}\quad \mathbb{A}\cong \left(\begin{smallmatrix}2&1\\1&2\end{smallmatrix}\right)\text{.}
\]
If $R = \z_2$, then $\mathbb{H}$ is an isotropic even unimodular lattice, whereas $\mathbb{A}$ is an anisotropic even unimodular lattice. For any odd prime $p$, let $\Delta_p$  be any nonsquare unit in $\z_p$. Any unexplained notation and terminology can be found in \cite{Ki} or \cite{OM}.

\section{The minimal rank of primitively $2$-universal $\z$-lattices}
\label{sec:rank}

It is well known in \cite{KKO98} that the minimal rank of $2$-universal $\z$-lattices is five, which implies that the minimal rank of primitively $2$-universal $\z$-lattices is at least five. The aim of this section is to show that the minimal rank of primitively $2$-universal $\z$-lattices is, in fact, six.

\begin{lem}\label{lem:aniso4}
For a prime $p$, let $L$ be an anisotropic $\z_p$-lattice. Let $L = L_0 \mathbin{\perp} \cdots \mathbin{\perp} L_t$ be a Jordan decomposition of $L$ such that $L_i = 0$ or $\fs(L_i) = p^i\z_p$. If $L_t\ne 0$, then any integer in $4p^{t+1}\z_p$ is not primitively represented by $L$.
\end{lem}

\begin{proof}
It is well-known that for any anisotropic unimodular $\z_p$-lattice $M$, if $Q(x)\in 4p\z_p$ for some $x\in M$, then $x\in pM$. For an $\alpha\in\z_p$, suppose that $4p^{t+1}\alpha$ is primitively represented by $L$. Then there is a primitive vector $x\in L$ such that $Q(x) = 4p^{t+1}\alpha$. Let $x = x_0 + \cdots + x_t$ for some $x_i\in L_i$ ($0\le i\le t$). Then there is a $k$ such that $x_k$ is primitive in $L_k$. If $p$ is odd, then $\ord_p (Q(x_k)) = k$, and thus
\[
 Q(x - x_k) \equiv -Q(x_k) \pmod{4p^{k+1}}\text{.}
\]
Hence, $Q(x - x_k) = - \epsilon^2 Q(x_k)$ for some unit $\epsilon\in\z_p$. Since $x-x_k\notin L_k$, the vector $(x-x_k) + \epsilon x_k$ is a nonzero isotropic vector. This is a contradiction.

Now, assume that $p = 2$. First, suppose that $L_k$ admits an orthogonal basis $e_1, \dots, e_m$. If $x_k = \alpha_1 e_1 + \cdots + \alpha_m e_m$, then $\ord_2 (Q(\alpha_j e_j)) = k$ for some $1\le j\le m$. Thus $Q(x - \alpha_j e_j) \equiv -Q(\alpha_j e_j) \Mod{4p^{k+1}}$, which implies $Q(x - \alpha_j e_j) = - \epsilon^2 Q(\alpha_j e_j)$ for some unit $\epsilon\in\z_2$. Since $x-\alpha_j e_j\notin \z_2 e_j$, the vector $(x - \alpha_j e_j) + \epsilon \alpha_j e_j$ is a nonzero isotropic vector. This is a contradiction. If $L_k$ does not admit an orthogonal basis, then $L_k\cong \mathbb{A}^{2^k}$ and $\ord_2 (Q(x_k)) = k+1$. Since $L_k$ represents every integer in $2^{k+1}\z_2^\times$, there is a vector $z\in L_k$ such that $Q(z) = -Q(x - x_k)$. Since $x-x_k\notin L_k$, $z + (x - x_k)$ is a nonzero isotropic vector. This is a contradiction.
\end{proof}

\begin{lem}\label{lem:anisocut}
For a positive integer $n$, let $U$ be an $n$-dimensional quadratic space over $\q_p$ and let $W$ be an anisotropic quadratic space over $\q_p$. Then any $\z_p$-lattice on $U\mathbin{\perp} W$ is not primitively $(n+1)$-universal.
\end{lem}

\begin{proof}
Let $L$ be any $\z_p$-lattice on $U\mathbin{\perp} W$ and $\ell$ be any $\z_p$-lattice on $U$. It is enough to show that there is an $a\in\z_p$ such that $L$ cannot primitively represent $\ell\mathbin{\perp}\langle a\rangle$. To do this, we show that there is an $a\in\z_p$ that is not primitively represented by $\sigma(\ell)^\perp$ for any primitive representation $\sigma : \ell \to L$.

We know that there are only finitely many possibilities of $\sigma(\ell)$ up to isometry of $L$ by \cite[Theorem~5.3.3]{Ki}. Therefore, for any possible primitive representation $\sigma : \ell \to L$, there are only finitely many $\sigma(\ell)^\perp$ up to isometry. For any such $\sigma(\ell)^\perp$, we have $\q_p \sigma(\ell)^\perp\cong W$, for $\q_p \sigma(\ell)\cong U$. Hence, $\sigma(\ell)^\perp$ is an anisotropic $\z_p$-lattice. Therefore by Lemma~\ref{lem:aniso4}, there exists an integer $r$ such that any integer in $p^r\z_p$ is not primitively represented by $\sigma(\ell)^\perp$. Now, if we choose any integer $a$ in the intersection of all possible $p^r\z_p$'s, then $\ell\mathbin{\perp}\langle a\rangle$ is not primitively represented by $L$. This completes the proof.
\end{proof}

\begin{lem}\label{lem:not5}
For any quinary $\z$-lattice $L$, there are infinitely many binary $\z$-lattices that are not primitively represented by $L$.
\end{lem}

\begin{proof}
Since $\q L$ represents any positive rational number, we may write $\q L = U \mathbin{\perp} \langle dL\rangle$ for some quaternary $\q$-space $U$. Since $U$ is a quaternary space of discriminant $1$, it follows that for each prime $p<\infty$, $U_p$ is isotropic if and only if $S_p U = (-1, -1)$. However, since $S_\infty U = 1 \ne (-1, -1)$, there is a prime $q$ such that $S_q U \ne (-1, -1)$ by Hilbert reciprocity law. This implies that $U_q$ is anisotropic. Hence, $L_q$ is not primitively $2$-universal from the above lemma. Therefore, there are infinitely many binary $\z$-lattices that are not primitively represented by $L$.
\end{proof}

\begin{thm}\label{thm:rankissix}
The minimal rank of primitively $2$-universal $\z$-lattices is $6$.
\end{thm}

\begin{proof}
By Lemma~\ref{lem:not5}, the minimal rank of primitively $2$-universal $\z$-lattices is at least six. Furthermore, since $I_6$ is primitively $2$-universal over $\z_p$ for any prime $p$ and is of class number one, $I_6$ is primitively $2$-universal (see Theorem~\ref{thm:pf41}). The theorem follows from this.
\end{proof}

\begin{defn}\label{defn:core}
The prime $q$ in the proof of the Lemma \ref{lem:not5} such that $S_q U \ne (-1, -1)$ \textup(or equivalently $U_q$ is anisotropic\textup) is called the \textbf{core prime} of $L$ \textup(see also \cite[Lemma~2.4]{Oh03}\textup).

\end{defn}

The existence of a core prime of a quinary $\z$-lattice will play a significant role when we determine binary $\z$-lattices that are primitively represented by a quinary $\z$-lattice in Sections \ref{sec:pf1} and \ref{sec:pf2}.

\section{Candidates of primitively $2$-universal senary $\z$-lattices}
\label{sec:201}

In this section, we find all candidates of primitively $2$-universal senary $\z$-lattices. 

A $\z$-sublattice $\z e_1 + \cdots + \z e_k$ of a $\z$-lattice $L$ is called a $k$-\textbf{section} of $L$ if there are vectors $e_{k+1}, \dots, e_n$ in $L$ such that $\{e_1, \dots, e_n\}$ is a Minkowski reduced basis for $L$. Recall that a $k$-section of $L$ is not unique in general.

Let $L$ be a primitively $2$-universal senary $\z$-lattice. We find all possible $k$-sections of $L$ for each $k = 1, \dots, 6$, inductively. To find all possible candidates of $(k+1)$-sections containing a specific $k$-section, we apply Lemma~\ref{lem:escal} repeatedly, which is quite well known (see \cite[Lemma~2.1]{LOY20}).

\begin{lem}\label{lem:escal}
Let $M$ and $N$ be $\z$-lattices of rank $m$ and $n$ respectively. Suppose that the ordered set $\{e_1, \dots, e_n\}$ is a Minkowski reduced basis for $N$. Suppose further that $M$ is represented by $N$, but not by the $k$-section $\z e_1 + \cdots + \z e_k$.
\begin{enumerate}
\item We have $Q(e_{k+1}) \le C_4(k+1) \max\{\mu_m(M), Q(e_k)\}$, where the constant $C_4(j)$ depending only on $j$ is defined in \cite{Ca} and $\mu_m(M)$ is the $m$-th successive minimum of $M$ \textup(see \cite[Theorem~3.1]{Ca}\textup).

\item Suppose further that $n\le m+4$ and $N$ is $m$-universal. Then for any $x = x_1 e_1 + \cdots + x_n e_n \in L$, we have $Q(x)\ge Q(e_j)$ for any $j$ such that $x_j\ne 0$. Also, we have $Q(e_{k+1})\le \mu_m(M)$.
\end{enumerate}
\end{lem}

\begin{proof}
(1) Suppose on the contrary that
\[
 Q(e_{k+1}) > C_4(k+1) \max\{\mu_m(M), Q(e_k)\}\text{.}
\]
Since $Q(e_{k+1}) \le C_4(k+1) \mu_{k+1}(N)$, we have
\[
 \mu_{k+1}(N) > \max\{\mu_m(M), Q(e_k)\}\text{.}
\]
Since we are assuming that $M$ is represented by $N$, the above inequality implies that
\[
 M \rightarrow \operatorname{span}_\q\{x\in L : Q(x)<\mu_{k+1}(N)\}\cap N = \z e_1 \oplus \cdots \oplus \z e_k\text{,}
\]
which is a contradiction.

\noindent (2) Since $I_m$ is represented by $N$, we have $N = I_m\mathbin{\perp} N'$ for some $\z$-sublattice $N'$ of $N$. Assume that the ordered set $\{e_{m+1}, \dots, e_n\}$ is a Minkowski reduced basis for $N'$. Since $\operatorname{rank} N'\le 4$, for any vector
\[
 x = x_{m+1} e_{m+1} + \cdots + x_n e_n\in N'\text{,}
\]
we have $Q(x)\ge Q(e_j)$ for any $j$ such that $x_j \ne 0$ (see \cite[Lemma~1.2 of Ch.~12]{Ca}). Hence, the former assertion holds trivially. Furthermore, it is well known that $Q(e_j) = \mu_j(N)$ for any $j$ with $1\le j\le n$. Now, on the contrary to the latter assertion, suppose that $Q(e_{k+1}) > \mu_m(M)$, then $\mu_{k+1}(N) > \mu_m(M)$. Since we are assuming that $M$ is represented by $N$, this inequality implies that
\[
 M \rightarrow \operatorname{span}_\q\{x\in L : Q(x)<\mu_{k+1}(N)\}\cap N \subseteq \z e_1 \oplus \cdots \oplus \z e_k\text{,}
\]
which is a contradiction.
\end{proof}

For a $\z$-lattice $M$, a binary $\z$-lattice $\ell$ is called a primitive binary exception of $M$ if $\ell$ is not primitively represented by $M$. If a $\z$-lattice $M$ is not primitively $2$-universal, we define the \textbf{truant} of $M$ to be, among all the primitive binary exceptions of $M$, the least isometry class with respect to the following total order in terms of Gram matrices:
\[
 \begin{pmatrix}a_1 & b_1 \\ b_1 & c_1\end{pmatrix} \prec \begin{pmatrix}a_2 & b_2 \\ b_2 & c_2\end{pmatrix} \quad \Leftrightarrow \quad \begin{cases}
  c_1 < c_2\text{,} & \text{ or}\\
  c_1 = c_2\text{ and }a_1 < a_2\text{,} & \text{ or}\\
  c_1 = c_2\text{, }a_1 = a_2\text{, and }b_1 < b_2\text{,}
 \end{cases}
\]
where we assume that all binary lattices are Minkowski reduced, that is, $0 \le 2b_i \le a_i \le c_i$ for any $i = 1, 2$.

Now, we find all candidates of a $k$-section of a primitively $2$-universal senary $\z$-lattice for each $k = 1, \dots, 6$, inductively. Since $I_2$ is the truant of any lattice of rank less than two, clearly the $2$-section of any primitively $2$-universal $\z$-lattice must be $I_2$. Since $I_2$ cannot primitively represent $\langle 1, 2\rangle$, we must have $1\le Q(e_3)\le 2$ by the last lemma. Thus, the $3$-section must be either $I_3$ or $I_2\mathbin{\perp}\langle 2\rangle$. It is easily seen that the truant of each $3$-section is $\ang{2, 2}$ and $\mathbb{A}$, respectively. From this, the $4$-section must be one of
\[
 I_4\text{,}\quad I_3\mathbin{\perp}\langle 2\rangle\text{,}\quad\text{or}\quad I_2\mathbin{\perp}\mathbb{A}\text{.}
\]
By repeating this and by removing candidates that have the truant, we finally obtain the following list of $201$ candidates of isometry classes of primitively $2$-universal (P$2$U for short) senary $\z$-lattices:

{\renewcommand{\arraystretch}{1.2}
\begin{longtable}{c|c|c|c}
\caption{The $201$ candidates of P$2$U senary $\z$-lattices}\label{tblr:201}\\\hline
 Type & $5$-section & Candidates & Possible $k$'s \endhead
 \hline
 A & $I_5$ & $I_5\mathbin{\perp} \langle k\rangle$ & $k=1$, $2$ \\\hline
 \multirow{2}*{B} & \multirow{2}*{$I_4\mathbin{\perp}\langle2\rangle$} & $I_4\mathbin{\perp}\langle2,k\rangle$ & $2\le k\le 5$\\* \cline{3-4}
 & & $I_4\mathbin{\perp}\left(\begin{smallmatrix}2&1\\1&k\end{smallmatrix}\right)$ & $k=2$, $3$, $5$, $6$\\\hline
 C & $I_4\mathbin{\perp}\langle3\rangle$ & $I_4\mathbin{\perp}\left(\begin{smallmatrix}3&1\\1&k\end{smallmatrix}\right)$ & $k=3$\\ \hline
 \multirow{3}*[-4pt]{D} & \multirow{3}*[-4pt]{$I_3\mathbin{\perp}\langle2,2\rangle$} & $I_3\mathbin{\perp}\langle2,2,k\rangle$ & $2\le k\le 6$\\* \cline{3-4}
 & & $I_3\mathbin{\perp}\langle2\rangle\mathbin{\perp}\left(\begin{smallmatrix}2&1\\1&k\end{smallmatrix}\right)$ & $3\le k\le 8$\\* \cline{3-4}
 & & $I_3\mathbin{\perp}\left(\begin{smallmatrix}2&0&1\\0&2&1\\1&1&k\end{smallmatrix}\right)$\rule[-9pt]{0pt}{24pt} & $k=3$, $5\le k\le 7$\\ \hline
 \multirow{2}*[-4pt]{E} & \multirow{2}*[-4pt]{$I_3\mathbin{\perp}\left(\begin{smallmatrix}2&1\\1&2\end{smallmatrix}\right)$} & $I_3\mathbin{\perp}\left(\begin{smallmatrix}2&1\\1&2\end{smallmatrix}\right)\mathbin{\perp}\langle k\rangle$ & $2\le k\le 5$, $k=7$, $8$ \\* \cline{3-4}
 & & $I_3\mathbin{\perp}\left(\begin{smallmatrix}2&1&1\\1&2&1\\1&1&k\end{smallmatrix}\right)$\rule[-9pt]{0pt}{24pt} & $2\le k\le 9$ \\ \hline
 F & $I_3\mathbin{\perp}\langle2,3\rangle$& $I_3\mathbin{\perp}\langle2\rangle\mathbin{\perp}\left(\begin{smallmatrix}3&1\\1&k\end{smallmatrix}\right)$ & $k=3$ \\ \hline
 \multirow{3}*[-6pt]{G} & \multirow{3}*[-6pt]{$I_3\mathbin{\perp}\left(\begin{smallmatrix}2&1\\1&3\end{smallmatrix}\right)$} & $I_3\mathbin{\perp}\left(\begin{smallmatrix}2&1\\1&3\end{smallmatrix}\right)\mathbin{\perp}\langle k\rangle$ & $k=3$, $4$, $6\le k\le 18$\\* \cline{3-4}
 & & $I_3\mathbin{\perp}\left(\begin{smallmatrix}2&1&0\\1&3&1\\0&1&k\end{smallmatrix}\right)$\rule[-9pt]{0pt}{24pt} & $3\le k\le 19$\\* \cline{3-4}
 & & $I_3\mathbin{\perp}\left(\begin{smallmatrix}2&1&1\\1&3&1\\1&1&k\end{smallmatrix}\right)$\rule[-9pt]{0pt}{24pt} & $3\le k\le 19$\\ \hline 
 \multirow{4}*[-12pt]{H} & \multirow{4}*[-12pt]{$I_2\mathbin{\perp}\left(\begin{smallmatrix}2&1\\1&2\end{smallmatrix}\right)\mathbin{\perp}\langle2\rangle$} & $I_2\mathbin{\perp}\left(\begin{smallmatrix}2&1\\1&2\end{smallmatrix}\right)\mathbin{\perp}\langle2,k\rangle$ & $3\le k\le 5$\\* \cline{3-4}
 & & $I_2\mathbin{\perp}\left(\begin{smallmatrix}2&1\\1&2\end{smallmatrix}\right)\mathbin{\perp}\left(\begin{smallmatrix}2&1\\1&k\end{smallmatrix}\right)$ & $k=2$, $4$, $5$\\* \cline{3-4}
 & & $I_2\mathbin{\perp}\left(\begin{smallmatrix}2&1&0&1\\1&2&0&1\\0&0&2&0\\1&1&0&k\end{smallmatrix}\right)$\rule[-12pt]{0pt}{30pt} & $3\le k\le 6$\\* \cline{3-4}
 & & $I_2\mathbin{\perp}\left(\begin{smallmatrix}2&1&0&1\\1&2&0&1\\0&0&2&1\\1&1&1&k\end{smallmatrix}\right)$\rule[-12pt]{0pt}{30pt} & $k=3$, $4$, $6$\\ \hline
 \multirow{3}*[-13pt]{I} & \multirow{3}*[-13pt]{$I_2\mathbin{\perp}\left(\begin{smallmatrix}2&1&1\\1&2&1\\1&1&2\end{smallmatrix}\right)$} & $I_2\mathbin{\perp}\left(\begin{smallmatrix}2&1&1\\1&2&1\\1&1&2\end{smallmatrix}\right)\mathbin{\perp}\langle k\rangle$\rule[-9pt]{0pt}{24pt} & $k=2$\\* \cline{3-4}
 & & $I_2\mathbin{\perp}\left(\begin{smallmatrix}2&1&1&1\\1&2&1&1\\1&1&2&0\\1&1&0&k\end{smallmatrix}\right)$\rule[-12pt]{0pt}{30pt} & $2\le k\le 5$\\* \cline{3-4}
 & & $I_2\mathbin{\perp}\left(\begin{smallmatrix}2&1&1&1\\1&2&1&1\\1&1&2&1\\1&1&1&k\end{smallmatrix}\right)$\rule[-12pt]{0pt}{30pt} & $2\le k\le 4$\\ \hline
 J & $I_2\mathbin{\perp}\left(\begin{smallmatrix}2&1\\1&2\end{smallmatrix}\right)\mathbin{\perp}\langle3\rangle$ & $I_2\mathbin{\perp}\left(\begin{smallmatrix}2&1\\1&2\end{smallmatrix}\right)\mathbin{\perp}\left(\begin{smallmatrix}3&1\\1&k\end{smallmatrix}\right)$ & $k=3$ \\ \hline
 \multirow{4}*[-18pt]{K} & \multirow{4}*[-18pt]{$I_2\mathbin{\perp}\left(\begin{smallmatrix}2&1&1\\1&2&1\\1&1&3\end{smallmatrix}\right)$} & $I_2\mathbin{\perp}\left(\begin{smallmatrix}2&1&1\\1&2&1\\1&1&3\end{smallmatrix}\right)\mathbin{\perp}\langle k\rangle$\rule[-9pt]{0pt}{24pt} & $3\le k\le 20$, $22\le k\le 24$\\* \cline{3-4}
 & & $I_2\mathbin{\perp}\left(\begin{smallmatrix}2&1&1&0\\1&2&1&0\\1&1&3&1\\0&0&1&k\end{smallmatrix}\right)$\rule[-12pt]{0pt}{30pt} & $3\le k\le 24$\\* \cline{3-4} 
 & & $I_2\mathbin{\perp}\left(\begin{smallmatrix}2&1&1&1\\1&2&1&1\\1&1&3&0\\1&1&0&k\end{smallmatrix}\right)$\rule[-12pt]{0pt}{30pt} & $3\le k\le 25$\\* \cline{3-4}
 & & $I_2\mathbin{\perp}\left(\begin{smallmatrix}2&1&1&1\\1&2&1&1\\1&1&3&1\\1&1&1&k\end{smallmatrix}\right)$\rule[-12pt]{0pt}{30pt} & $3\le k\le 25$\\ \hline
\end{longtable}}

We refer to any of the above candidates by the expression such as (type) or (type)$_k$. For instance, $89$ candidates in the last four rows are called type K. Among them, $21$, $22$, $23$, and $23$ candidates in each of four rows are of types K\textsuperscript{i}, K\textsuperscript{ii}, K\textsuperscript{iii}, and K\textsuperscript{iv}, respectively. For instance, a lattice of type K\textsuperscript{iv} is any $\z$-lattice of the form $I_2\mathbin{\perp}\left(\begin{smallmatrix}2&1&1&1\\1&2&1&1\\1&1&3&1\\1&1&1&k\end{smallmatrix}\right)$. Finally, K$^{\mbox{\scriptsize iv}}_3$ stands for the lattice $I_2\mathbin{\perp}\left(\begin{smallmatrix}2&1&1&1\\1&2&1&1\\1&1&3&1\\1&1&1&3\end{smallmatrix}\right)$.

Moreover, when we refer to each of the candidates in Table~\ref{tblr:201}, we always assume that the basis $e_1, \dots, e_6$ corresponds exactly to the Gram matrix given in the table. For instance, when we consider the $\z$-lattice K$^{\mbox{\scriptsize iv}}_3 = \z e_1 + \dotsb + \z e_6$, we always assume that
\[
 (B(e_i, e_j)) = \left(\begin{smallmatrix}1&0&0&0&0&0\\0&1&0&0&0&0\\0&0&2&1&1&1\\0&0&1&2&1&1\\0&0&1&1&3&1\\0&0&1&1&1&3\end{smallmatrix}\right)\text{.}
\]
From now on, by abusing of terminology, the $s$-section of $\z e_1 + \cdots + \z e_6$ always implies the $\z$-sublattice $\z e_1 + \cdots + \z e_s$ for any $s = 1$, $2$, \dots, $6$.

\section{The proof of primitive $2$-universality (ordinary cases)}
\label{sec:pf1}

In this section, we prove that some candidates given in Table~\ref{tblr:201} are, in fact, primitively $2$-universal. Let $L = \z e_1 + \dotsb \z e_6$ be one of candidates of primitively $2$-universal $\z$-lattices, where the corresponding symmetric matrix $(B(e_i, e_j))$ is given in Table~\ref{tblr:201}. We prove the primitive $2$-universality of $L$ for the cases when $L$ itself or the $5$-section of $L$ is of class number one. Since the key ingredients of the proofs, which we will explain more precisely later, are quite similar to each other, we only provide the proofs of some representative cases. For the complete proofs, one may see the second author's dissertation \cite{Yo}.

\subsection{Case 1: $L$ has class number one}
\label{ssec:41}

As explained in the introduction, if a $\z$-lattice is of class number one, then it primitively represents any $\z$-lattice that is primitively represented over $\z_p$ for any prime $p$. Hence, if $L$ is of class number one, then $L$ is primitively $2$-universal if and only if $L_p$ primitively represents all binary $\z_p$-lattices for any prime $p$.

\begin{thm}\label{thm:pf41}
If $L$ has class number one, then $L$ is primitively $2$-universal. In fact, there are exactly $10$ such $\z$-lattices in Table~\ref{tblr:201}.
\end{thm}

\begin{proof}
One may easily show that
\[
 L \cong \text{A$_1$, A$_2$, B$^{\mbox{\scriptsize i}}_2$, B$^{\mbox{\scriptsize ii}}_2$, B$^{\mbox{\scriptsize ii}}_3$, D$^{\mbox{\scriptsize iii}}_3$, E$^{\mbox{\scriptsize ii}}_2$, E$^{\mbox{\scriptsize ii}}_3$, or I$^{\mbox{\scriptsize ii}}_3$.}
\]
Note that $dL = 1$, $2$, $4$, $3$, $5$, $8$, $4$, $7$, $4$, and $8$, respectively. Since $\mathbb{H}\mathbin{\perp}\mathbb{H}$ is primitively represented by $L_p$ for any odd prime $p$, $L_p$ is primitively $2$-universal. One may directly show that $L_2$ is also primitively $2$-universal.
\end{proof}

\begin{rmk}
In fact, for senary lattices, Lemma~\ref{thm:pf41} generalizes Budarina's result \cite{Bu11}, where $L$ is required to be of class number one and to be of squarefree odd discriminant. One may verify from the proof that there are exactly four such cases out of our $201$ candidates.
\end{rmk}

\subsection{Case 2: the $5$-section of $L$ has class number one and splits $L$ orthogonally}
\label{ssec:42}

In the remaining of this section, we consider the case when the $5$-section of $L$, say $M$, has class number one. Since $M$ is a primitive sublattice of $L$, $L$ primitively represents any $\z$-lattice that is primitively represented by $M$. Hence, if $M$ is of class number one, then $L$ primitively represents any $\z$-lattice that is locally primitively represented by $M$. Note that by Lemma~\ref{lem:not5}, $M_q$ is not primitively $2$-universal for some prime $q$, and hence there are infinitely many $\z$-lattices that are not primitively represented by $M$.

Recall that any core prime $q$ of $M$ satisfies $S_q U \ne (-1, -1)$, where $\q M \cong U\mathbin{\perp} \langle dM \rangle$ (see Definition \ref{defn:core}). One may easily check that the $5$-section of $L$ whose type is not of H has class number one, and the $5$-section $I_2\mathbin{\perp}\mathbb{A}\mathbin{\perp} \langle 2 \rangle$ of $L$ with type H has class number two. Note that the genus mate of this lattice is $I_4\mathbin{\perp} \langle 6 \rangle$.

Hereafter, $\alpha$, $\beta$ denote integers in $\z_p$, and $\epsilon$, $\delta$ denote units in $\z_p$, unless stated otherwise, where the prime $p$ could be easily verified from the context.

\begin{lem}\label{lem:5excep}
For the $5$-section $M$ of each type given in Table~\ref{tblr:5excep}, the core prime $q$ of $M$ and local structures over $\z_q$ of any binary $\z$-lattice $\ell$ that is not primitively represented by $M$ are given as in the table.
\end{lem}

\begin{table}[ht]
\caption{The core prime and the local structures}
\label{tblr:5excep}
\renewcommand{\arraystretch}{1.2}
\begin{tabular}{c|c|c|l}
 \hline
 Type & $M$ & $q$ & \hfil Local structures\\\hline
 A & $I_5$ & $2$ & $\ell_2\cong \langle 1, 8\alpha \rangle$ or $\fn (\ell_2)\subseteq 4\z_2$\\\hline
 B & $I_4\mathbin{\perp} \langle 2 \rangle$ & $2$ & $\ell_2\cong \langle 2, 16\alpha \rangle$ or $\fn (\ell_2)\subseteq 8\z_2$\\\hline
 D & $I_3\mathbin{\perp} \langle 2,2 \rangle$ & $2$ &
 \begin{tabular}{@{}l@{}}
  $\ell_2\cong \langle 1, 16\alpha \rangle$, $\langle 4, 16\alpha \rangle$, $\langle 20, 16\alpha \rangle$, or\\[-2pt]
  $\fn (\ell_2)\subseteq 16\z_2$
 \end{tabular}\\\hline
 E & $I_3\mathbin{\perp}\left(\begin{smallmatrix}2&1\\1&2\end{smallmatrix}\right)$ & $3$ & $\ell_3\cong \langle 3, 9\alpha \rangle$ or $\fs (\ell_3)\subseteq 9\z_3$\\\hline
 G & $I_3\mathbin{\perp}\left(\begin{smallmatrix}2&1\\1&3\end{smallmatrix}\right)$ & $5$ & $\ell_5\cong \langle 5, 25\alpha \rangle$ or $\fs (\ell_5)\subseteq 25\z_5$\\\hline
 I & $I_2\mathbin{\perp}\left(\begin{smallmatrix}2&1&1\\1&2&1\\1&1&2\end{smallmatrix}\right)$ & $2$ &
 \begin{tabular}{@{}l@{}}
  $\ell_2\cong\langle 1, 32\alpha \rangle$, $\langle 5, 16\epsilon \rangle$, $\langle 4, 32\alpha \rangle$, or\\[-2pt]
  $\fn (\ell_2)\subseteq 16\z_2$
 \end{tabular}\\\hline
 K & $I_2\mathbin{\perp}\left(\begin{smallmatrix}2&1&1\\1&2&1\\1&1&3\end{smallmatrix}\right)$\rule[-9pt]{0pt}{24pt} & $7$ & $\ell_7\cong \langle 7, 49\alpha \rangle$ or $\fs (\ell_7)\subseteq 49\z_7$\\\hline
\end{tabular}
\end{table}

\begin{proof}
One may easily verify that the prime $q$ given in Table~\ref{tblr:5excep} is the only core prime for each $5$-section $M$, and any binary $\z_p$-lattice is primitively represented by $M_p$ for any prime $p\ne q$. Hence, a binary $\z$-lattice $\ell$ is primitively represented by $M$ if and only if $\ell_q$ is primitively represented by $M_q$.

Since all the other cases can be proved in similar manners, we only provide the proofs of the cases when $L$ is of types I and K.

First, assume that $L$ is of type I. Note that $M_2\cong \mathbb{H}\mathbin{\perp} N$, where $N\cong \ang{3,7,12}$ and
\[
 Q^\ast(N) = \{3,7\}(\z_2^\times)^2\cup 2\z_2^\times\cup \{12, 20, 28\}(\z_2^\times)^2\cup 8\z_2^\times\text{.}
\]
Hence, $M_2$ primitively represents all binary lattices of the form $\ang{\alpha, \theta}$, $\mathbb{H}^{2^a}$, and $\mathbb{A}^{2^a}$, where $\theta\in Q^\ast(N)$ and $0\le a\le 2$. Moreover, $M_2$ primitively represents $\ang{\theta, \theta'}$, where $\theta$, $\theta'\in \{1, 5, 4\}$, $\ang{1, 16\epsilon}$, $\ang{5, 32\alpha}$, and $\ang{4, 16\epsilon}$.

Now, assume that $L$ is of type K. Note that $M_7\cong \mathbb{H}\mathbin{\perp} N$, where $N\cong \ang{1,1,7\cdot\Delta_7}$ and $Q^\ast(N) = \{1,\Delta_7,7\cdot\Delta_7\}(\z_7^\times)^2$. In this case, the lemma follows directly from $\ang{7,7}\cong \ang{7\cdot\Delta_7,7\cdot\Delta_7}$. This completes the proof.
\end{proof}

\begin{rmk}
In Table~\ref{tblr:5excep}, if a binary $\z$-lattice $\ell$ is not primitively represented by $M$, then the local structures of $\ell$ necessarily satisfies one of the conditions given in the fourth column in the same row. However, we do not know whether or not any binary $\z$-lattice satisfying one of the local structures is necessarily not primitively represented by $M$.
\end{rmk}

We first complete the proof of the case when the $5$-section splits $L$ orthogonally.

\begin{thm}\label{thm:pf42}
Suppose that $L$ has class number at least two. If $M$ is of class number one and splits $L$ orthogonally, then $L$ is primitively $2$-universal. In fact, there are exactly $51$ such lattices in Table~\ref{tblr:201}.
\end{thm}

\begin{proof}
By the above lemma, $L$ is of type
\[
\text{B\textsuperscript{i}(except B$^{\mbox{\scriptsize i}}_2$), D\textsuperscript{i}, E\textsuperscript{i}, G\textsuperscript{i}, I\textsuperscript{i}, or K\textsuperscript{i}.}
\]
Let $\ell\cong \left(\begin{smallmatrix}a&b\\b&c\end{smallmatrix}\right)$ be a $\z$-lattice which is not primitively represented by $M$. We assume that $\ell$ is Minkowski reduced, that is, $0\le 2b\le a\le c$. To show that $\ell$ is primitively represented by $L$, we may consider two $\z$-lattices
\[
 \ell' \cong \begin{pmatrix}a-k&b\\b&c\end{pmatrix}\text{,}\qquad \ell'' \cong \begin{pmatrix}a&b\\b&c-k\end{pmatrix}\text{.}
\]
If either $\ell'$ or $\ell''$ is primitively represented by $M$, then clearly $\ell$ is primitively represented by $L\cong M\mathbin{\perp}\ang{k}$. Moreover, $\ell'$ ($\ell''$) is primitively represented by $M$ if and only if $\ell'_q$ ($\ell''_q$, repsectively) is primitively represent by $M_q$ for the core prime $q$ of $M$.

Since all the other cases can be proved in similar manners, we only provide the proofs of the cases when $L$ is of types D\textsuperscript{i} or G\textsuperscript{i}.

First, assume that $L$ is of type D\textsuperscript{i}. By Lemma~\ref{lem:5excep}, we may assume that
\[
 \ell_2\cong\ang{1, 16\alpha}\text{,}\quad \ang{4, 16\alpha}\text{,}\quad \ang{20, 16\alpha}\text{,} \quad \text{or} \quad \fn (\ell_2)\subseteq 16\z_2\text{.}
\]
Suppose that $k = 3$ or $5$. Assume that $a\not\equiv 0\Mod{16}$. Since $d\ell''_2 \not\equiv 0\Mod{16}$, $\ell''_2$ is primitively represented by $M_2$. Hence, $\ell''$ is primitively represented by $M$ if $c\ge 7$, in which case $\ell''$ is positive definite. Now, assume that $a\equiv 0\Mod{16}$. Since $a-k$ is odd, $\ell'_2$ is split by $\ang{a-k}$. Furthermore, since $a-k \not\equiv 1\Mod8$, $\ell'_2$ is primitively represented by $M_2$. Hence, $\ell'$ is primitively represented by $M$. One may directly check that $\ell$ is primitively represented by $L$ if $c\le 6$.

Now, suppose that $k = 2$ or $6$. Assume that $a\not\equiv 0\Mod8$. Since $d\ell''_2 \not\equiv 0\Mod{16}$, $\ell''_2$ is primitively represented by $M_2$. Hence, $\ell''$ is primitively represented by $M$ if $c\ge 9$, in which case $\ell''$ is positive definite. Now, assume that $c$ is odd. Since $c \equiv 1\Mod8$, $\ell''_2$ is split by $\ang{c-k}$. Furthermore, since $c-k \not\equiv 1\Mod8$, $\ell''$ is primitively represented by $M$ if $c\ge 9$. Finally, assume that $a\equiv 0\Mod8$ and $c$ is even. Since $\ell_2$ is not unimodular, we must have $\fs (\ell_2) \subseteq 4\z_2$, which implies
\[
 c-k \equiv 2\Mod4 \quad\text{and}\quad \fs (\ell''_2) = 2\z_2\text{.}
\]
Hence, $\ell''$ is primitively represented by $M$ if $c\ge 9$. One may directly check that $\ell$ is primitively represented by $L$ if $c\le 8$.

Finally, suppose that $k = 4$. Assume that $a\not\equiv 0\Mod4$. Since $d\ell''_2 \not\equiv 0\Mod{16}$, $\ell''_2$ is primitively represented by $M_2$. Hence, $\ell''$ is primitively represented by $M$ if $c\ge 6$, in which case $\ell''$ is positive definite. Now, assume that $c$ is odd. Since $c \equiv 1\Mod8$, $\ell''_2$ is split by $\ang{c-4}$. Furthermore, since $c-4 \not\equiv 1\Mod8$, $\ell''$ is primitively represented by $M$ if $c\ge 6$. Hence, we may assume that $a\equiv 0\Mod4$ and $c$ is even. Since $\ell_2$ is not unimodular, we have
\[
 \fs (\ell_2) \subseteq 4\z_2 \quad\text{and}\quad d\ell_2 \equiv 0\Mod{64}\text{.}
\]
If $a\not\equiv 0\Mod{16}$, then $d\ell''_2 \not\equiv 0\Mod{64}$; hence $\ell''$ is primitively represented by $M$ if $c\ge 6$. Now, assume that $a\equiv 0\Mod{16}$. Since $\fs (\ell'_2) = 4\z_2$, $\ell'_2$ is split by $\ang{a-4}$. Furthermore, since $a-4 \equiv -4\Mod{16}$, $\ell'_2$ is primitively represented by $M_2$. Hence, $\ell'$ is primitively represented by $M$ for $a\ge 16$. One may directly check that $\ell$ is primitively represented by $L$ if $c\le 5$.

Next, assume that $L$ is of type G\textsuperscript{i}. By Lemma~\ref{lem:5excep}, we may assume that
\[
 \ell_5\cong \ang{5, 25\alpha} \quad \text{or} \quad \fs (\ell_5)\subseteq 25\z_5\text{.}
\]
Suppose that $k\ne 10$ or $15$. Since $\fs (\ell''_5) = \z_5$, $\ell''_5$ is primitively represented by $M_5$. Hence, $\ell''$ is primitively represented by $M$ if $c\ge 25$, in which case $\ell''$ is positive definite. One may directly check that $\ell$ is primitively represented by $L$ if $c\le 24$.

Now, suppose that $k = 10$ or $15$. Assume that $a\not\equiv 0\Mod{25}$. Since $d\ell''_5 \not\equiv 0\Mod{125}$, $\ell''_5$ is primitively represented by $M_5$. Hence, $\ell''$ is primitively represented by $M$ if $c\ge 21$, in which case $\ell''$ is positive definite. Now, assume that $a\equiv 0\Mod{25}$. Since $\fs (\ell'_5) = 5\z_5$, $\ell'_5$ is split by $\ang{a-k}$. Furthermore, since $a-k\equiv 15$ or $10\Mod{25}$, $\ell'_5$ is primitively represented by $M_5$. Hence, $\ell'$ is primitively represented by $M$ since $a\ge 25$. One may directly check that $\ell$ is primitively represented by $L$ if $c\le 20$.
\end{proof}

\begin{rmk}
If $L\cong$ \textup{C}$_3$ or \textup{F}$_3$, then we may take $M \cong I_3\mathbin{\perp}\left(\begin{smallmatrix}3&1\\1&3\end{smallmatrix}\right)$, which is a primitive sublattice of $L$. If $L\cong$ \textup{H$^{\mbox{\scriptsize iii}}_3$}, then we may take $M \cong I_2\mathbin{\perp}\left(\begin{smallmatrix}2&1&1\\1&2&1\\1&1&3\end{smallmatrix}\right)$, which is a primitive sublattice of $L$. Since the class number of $M$ is one and $M$ splits $L$ orthogonally, the proofs of these three candidates are quite similar to that of the above theorem.
\end{rmk}

Let $R = \z$ or $\z_p$ for some prime $p$. Let $O$ be an $R$-lattice and let $\mathfrak B=\{e_1, \dots, e_n\}$ be the fixed (ordered) basis for the $R$-lattice $O$.
We define
\[
O'= R\begin{bmatrix}
  a_{11} & \dots & a_{1n}\\
  \vdots & & \vdots\\
  a_{m1} & \dots & a_{mn}
 \end{bmatrix}
\]
by the $R$-sublattice of $O$ generated by $m$ vectors $a_{11}e_1 + \cdots + a_{1n}e_n, \dots,\allowbreak a_{m1}e_1 + \cdots + a_{mn}e_n$.
Note that if the rank of $O'$ is $m$, then the symmetric matrix corresponding to $O'$ is that
\[
M_{O'}= \begin{pmatrix}
  a_{11} & \dots & a_{1n}\\
  \vdots & & \vdots\\
  a_{m1} & \dots & a_{mn}
 \end{pmatrix} M_O  \begin{pmatrix}
  a_{11} & \dots & a_{m1}\\
  \vdots & & \vdots\\
  a_{1n} & \dots & a_{mn}
 \end{pmatrix}\text{.}
\]
When only the corresponding symmetric matrix $M_O$ is given instead of the basis for $O$, we assume that $\mathfrak B$ is the basis for $O$ such that $(B(e_i,e_j))=M_O$.

\subsection{Case 3: the $5$-section of $L$ has class number one and does not split $L$ orthogonally}
\label{ssec:43}

Recall that we are assuming that $L$ is one of  $201$ candidates of primitively $2$-universal senary $\z$-lattices given in Table~\ref{tblr:201}, and $\mathfrak B=\{e_1,\dots,e_6\}$ is the basis for $L$ such that the corresponding symmetric matrix $(B(e_i,e_j))$ is given in Table~\ref{tblr:201}. Furthermore, $M=\z e_1+\dots+\z e_5$ is the $5$-section of $L$.

\begin{lem}\label{lem:357super}
Let $L$ and $N$ be $\z$-lattices given as follows:
\begin{enumerate}
\item The $\z$-lattice $L$ is of type \textup{E\textsuperscript{ii}}, and $N \cong qI_3\mathbin{\perp}\mathbb{A}$.
\item The $\z$-lattice $L$ is of type \textup{G\textsuperscript{ii}} or \textup{G\textsuperscript{iii}}, and $N \cong qI_3\mathbin{\perp}\left(\begin{smallmatrix}2&1\\1&3\end{smallmatrix}\right)$.
\item The $\z$-lattice $L$ is of type \textup{K\textsuperscript{ii}}, \textup{K\textsuperscript{iii}} or \textup{K\textsuperscript{iv}}, and $N \cong qI_2\mathbin{\perp}\left(\begin{smallmatrix}3&1&1\\1&5&-2\\1&-2&5\end{smallmatrix}\right)$.
\end{enumerate}
Here, $q$ is the core prime of the $5$-section $M$ of $L$.
Let $\ell \cong \left(\begin{smallmatrix}a&b\\b&c\end{smallmatrix}\right)$ be a binary $\z$-lattice.
If $\ell' \cong \left(\begin{smallmatrix}qa-ds^2 & qb-dst\\qb-dst & qc-dt^2\end{smallmatrix}\right)$ is positive definite and is primitively represented by $N$ for some integers $s$ and $t$, where $d = dL$ is the discriminant of $L$, then $\ell$ is primitively represented by $L$.
\end{lem}

\begin{proof}
One may easily show that there is a representation $\psi: L^q \to N\mathbin{\perp} \langle d\rangle$ such that $\psi(L^q)\cap N=M^q$.

Since all the other cases can be proved in similar manners, we only provide the proof of the case when $L$ is of type K\textsuperscript{iv}. Note that $M \cong I_2 \mathbin{\perp} \left(\begin{smallmatrix}2&1&1\\1&2&1\\1&1&3\end{smallmatrix}\right)$ and $q=7$. Since $N$ primitively represents $\ell'$, there are integers $c_i$'s and $d_i$'s for $i=1,\dots,5$ such that the primitive sublattice of $N$
\[
 \z\begin{bmatrix}c_1 & c_2 & c_3 & c_4 & c_5\\d_1 & d_2 & d_3 & d_4 & d_5\end{bmatrix} \cong \ell' \cong \begin{pmatrix}7a-(7k-5)s^2 & 7b-(7k-5)st\\7b-(7k-5)st & 7c-(7k-5)t^2\end{pmatrix}\text{,}
\]
where $\left(\begin{smallmatrix}c_1 & c_2 & c_3 & c_4 & c_5\\d_1 & d_2 & d_3 & d_4 & d_5\end{smallmatrix}\right)$ is a primitive matrix. Note that
\[
 \left(\begin{smallmatrix}3&1&1\\1&5&-2\\1&-2&5\end{smallmatrix}\right) \equiv 3\left(\begin{smallmatrix}1\\-2\\-2\end{smallmatrix}\right)\begin{pmatrix}1&-2&-2\end{pmatrix}\Mod7\text{.}
\]
Thus, we have
\[
 \left.\begin{aligned}
  3(c_3 - 2c_4 - 2c_5)^2 & \equiv 5s^2\\
  3(c_3 - 2c_4 - 2c_5)(d_3 - 2d_4 - 2d_5) & \equiv 5st\\
  3(d_3 - 2d_4 - 2d_5)^2 & \equiv 5t^2
 \end{aligned}\right\} \pmod7\text{.}
\]
Hence, after replacing $(s, t)$ by $(-s, -t)$, if necessary, we may assume that
\[
 c_3 - 2c_4 - 2c_5 + 2s \equiv d_3 - 2d_4 - 2d_5 + 2t \equiv 0\pmod7\text{.}
\]
Therefore, there are integers $a_3$, $a_4$, $a_5$, $b_3$, $b_4$, and $b_5$ satisfying
\[
 \begin{pmatrix}c_3 & d_3\\c_4 & d_4\\c_5 & d_5\\s & t\end{pmatrix} = \begin{pmatrix}-1&0&2&0\\2&1&1&1\\1&-1&0&0\\0&0&0&1\end{pmatrix} \begin{pmatrix}a_3 & b_3\\a_4 & b_4\\a_5 & b_5\\s & t\end{pmatrix}\text{.}
\]
Now, consider the sublattice
\[
 O = \z\begin{bmatrix}c_1 & c_2 & a_3 & a_4 & a_5 & s\\d_1 & d_2 & b_3 & b_4 & b_5 & t\end{bmatrix}
\]
of $L$. Since
\begin{multline*}
 \begin{pmatrix}a_3 & a_4 & a_5 & s\\b_3 & b_4 & b_5 & t\end{pmatrix} \begin{pmatrix}14&7&7&7\\7&14&7&7\\7&7&21&7\\7&7&7&7q\end{pmatrix} \begin{pmatrix}a_3 & b_3\\a_4 & b_4\\a_5 & b_5\\s & t\end{pmatrix}\\
 = \begin{pmatrix}c_3 & c_4 & c_5 & s\\d_3 & d_4 & d_5 & t\end{pmatrix} \begin{pmatrix}3&1&1&0\\1&5&-2&0\\1&-2&5&0\\0&0&0&7q-5\end{pmatrix} \begin{pmatrix}c_3 & d_3\\c_4 & d_4\\c_5 & d_5\\s & t\end{pmatrix}\text{,}
\end{multline*}
the $\z$-lattice $O$ is isometric to $\ell$. Now, since
\[
 \begin{pmatrix}c_1 & c_2 & c_3 & c_4 & c_5\\d_1 & d_2 & d_3 & d_4 & d_5\end{pmatrix} = \begin{pmatrix}c_1 & c_2 & -a_3 + 2a_5 & 2a_3 + a_4 + a_5 + s & a_3 - a_4\\d_1 & d_2 & -b_3 + 2b_5 & 2b_3 + b_4 + b_5 + t & b_3 - b_4\end{pmatrix}
\]
is a primitive matrix, so is
\[
 \begin{pmatrix}c_1 & c_2 & -a_3 + 2a_5 & 2a_3 + a_4 + a_5 + s & a_3 - a_4 & a_5\\d_1 & d_2 & -b_3 + 2b_5 & 2b_3 + b_4 + b_5 + t& b_3 - b_4 & b_5\end{pmatrix}\text{.}
\]
Therefore, the matrix
\[
 \begin{pmatrix}c_1 & c_2 & a_3 & a_4 & a_5 & s\\d_1 & d_2 & b_3 & b_4 & b_5 & t\end{pmatrix}
\]
is primitive, which implies that $O$ is a primitive sublattice of $L$. This completes the proof.
\end{proof}

\begin{thm}\label{thm:pf43}
If $L$ is of type
\[
	 \text{\textup{E\textsuperscript{ii}, G\textsuperscript{ii}, G\textsuperscript{iii}, K\textsuperscript{ii}, K\textsuperscript{iii}, or K\textsuperscript{iv}},}
\]
then $L$ is primitively $2$-universal. There are exactly $110$ such $\z$-lattices in Table~\ref{tblr:201}.
\end{thm}

\begin{proof}
Let $\ell\cong \left(\begin{smallmatrix}a&b\\b&c\end{smallmatrix}\right)$ ($0 \le 2b\le a\le c$) be a $\z$-lattice which is not primitively represented by $M$. Since all the other cases can be proved in similar manners, we only provide the proofs of the cases when $L$ is of type K. Note that 
\[
 \det(\text{K\textsuperscript{ii}})=7k-3\text{,}\quad \det(\text{K\textsuperscript{iii}})=7k-6\text{,}\quad \text{and}\quad \det(\text{K\textsuperscript{iv}})=7k-5\text{.}
\]
Assume that $L$ is of type K. By Lemma~\ref{lem:5excep}, we may assume that
\[
 \ell_7\cong\ang{7, 49\alpha} \quad \text{or} \quad \fs (\ell_7)\subseteq 49\z_7\text{.}
\]
Observe that $N = 7I_2\mathbin{\perp}\left(\begin{smallmatrix}3&1&1\\1&5&-2\\1&-2&5\end{smallmatrix}\right)$ is of class number one, and that $7$ is the only core prime of $N$. Furthermore, $N_7\cong \mathbb{H}^7\mathbin{\perp} O$, where $O = \ang{\Delta_7, 7, 7}$. Thus $N_7$ primitively represents all binary $\z_7$-lattices of the form $\ang{7\alpha, \theta}$, where $\theta\in Q^\ast(O) = \{\Delta_7, 7, 7\cdot\Delta_7\}(\z_7^\times)^2$. Hence, $N_7$ primitively represents
\[
 \ell' \cong \left(\begin{smallmatrix}7a & 7b\\7b & 7c - (7k-3)\end{smallmatrix}\right)\text{,}\quad \left(\begin{smallmatrix}7a & 7b\\7b & 7c - (7k-6)\end{smallmatrix}\right)\text{,}\quad \text{and}\quad \left(\begin{smallmatrix}7a & 7b\\7b & 7c - (7k-5)\end{smallmatrix}\right)\text{.}
\]
Therefore $N$ primitively represents $\ell'$ if it is positive definite. Hence, by Lemma \ref{lem:357super}, $\ell$ is primitively represented by $L$ if $c \ge 33$. One may directly check that $\ell$ is primitively represented by $L$ if $c\le 32$.
\end{proof}

\begin{rmk}
In fact, we do not use the fact that $M$ is the $5$-section of $L$ in the above theorem. 
If $L\cong$ \textup{D$^{\mbox{\scriptsize ii}}_3$}, \textup{H$^{\mbox{\scriptsize iv}}_3$}, or \textup{I$^{\mbox{\scriptsize iii}}_3$}, then we may take a primitive sublattice $M$ of $L$ as in the following table. Then $q$ is the only core prime of $M$, and one may apply Lemma~\ref{lem:357super} for each pair of $L$ and $N$ in the table. Therefore, the proofs of primitive $2$-universalities of these three candidates are quite similar to Theorem~\ref{thm:pf43}.
\end{rmk}

\begin{table}[ht]
\caption{The core prime of $M$ and the choice of $N$}
\label{tblr:rmk}
\renewcommand{\arraystretch}{1.2}
\begin{tabular}{c|c|c|c}
 \hline
 $L$ & $M$ & $q$ & $N$\\\hline
 D$^{\mbox{\scriptsize ii}}_3$ & $I_3\mathbin{\perp}\left(\begin{smallmatrix}2&1\\1&3\end{smallmatrix}\right)$ & $5$ & $5I_3\mathbin{\perp}\left(\begin{smallmatrix}2&1\\1&3\end{smallmatrix}\right)$\\\hline
 H$^{\mbox{\scriptsize iv}}_3$ or I$^{\mbox{\scriptsize iii}}_3$ & $I_2\mathbin{\perp}\left(\begin{smallmatrix}2&1&1\\1&2&1\\1&1&3\end{smallmatrix}\right)$\rule[-9pt]{0pt}{24pt} & $7$ & $7I_2\mathbin{\perp}\left(\begin{smallmatrix}3&1&1\\1&5&-2\\1&-2&5\end{smallmatrix}\right)$\\\hline
\end{tabular}
\end{table}

Summing up all, we have proved the primitive $2$-universalities of $175$ $\z$-lattices among $201$ candidates, and the primitive $2$-universalities of $26$ candidates remain unproven.

\section{The proof of primitive $2$-universality (exceptional cases)}
\label{sec:pf2}

Let $L$ be one of the remaining $26$ candidates of primitively $2$-universal $\z$-lattices which we do not consider in Section~\ref{sec:pf1}. We try to find quinary or quaternary primitive sublattices of $L$ which have class number one or two to prove primitive $2$-universality of $L$, in each exceptional case. Since the key ingredients of the proofs, which we will explain more precisely later, are quite similar to each other, we only provide the proofs of some representative cases. For the complete proofs, one may see the second author's dissertation \cite{Yo}. At the end of this section, we provide some essential data needed for computations.

\subsection{Case 4: Type H lattices}
\label{ssec:5h}

In this subsection, we prove the primitive $2$-universalities of the remaining type H lattices. The main obstacle for type H lattices is that the $5$-section $I_2\mathbin{\perp}\mathbb{A}\mathbin{\perp}\ang{2}$ of $L$ is of class number two, and the genus mate $I_4\mathbin{\perp}\ang{6}$ is not represented by $L$. The following lemma gives some information on binary $\z$-lattices that are primitively represented by the $5$-section of $L$, though it has class number two.

\begin{lem}\label{lem:superd2h}
Let $\ell \cong \left(\begin{smallmatrix}a&b\\b&c\end{smallmatrix}\right)$ be a binary $\z$-lattice and let $m$ and $k\ge 2$ be positive integers. Suppose that
\[
 \begin{pmatrix}2a-(2k-1)s^2 & 2b-(2k-1)st\\2b-(2k-1)st & 2c-(2k-1)t^2\end{pmatrix}
\]
is positive definite and is primitively represented by $2I_m\mathbin{\perp}\ang{4,1}$ for some integers $s$ and $t$. Then the binary $\z$-lattice $\ell$ is primitively represented by $I_m\mathbin{\perp}\ang{2}\mathbin{\perp}\left(\begin{smallmatrix}2&1\\1&k\end{smallmatrix}\right)$.
\end{lem}

\begin{proof}
Since the proof is quite similar to that of Lemma~\ref{lem:357super}, it is left to the readers.
\end{proof}

\begin{lem}\label{lem:h}
If a binary $\z$-lattice $\ell$ is not primitively represented by the $\z$-lattice $M = I_2\mathbin{\perp}\ang{2}\mathbin{\perp}\mathbb{A}$, then either $\ell_2\cong \ang{6, 16\alpha}$ or $\fn (\ell_2)\subseteq 8\z_2$.
\end{lem}

\begin{proof}
Fix a basis for $M$ corresponding to the Gram matrix in the statement of the lemma. Let $\ell\cong \left(\begin{smallmatrix}a&b\\b&c\end{smallmatrix}\right)$ ($0\le 2b\le a\le c$) be a binary $\z$-lattice which does not satisfy the conclusion. Note that $M$ primitively represents $\ang{1, 1, 2, 2}$, which has class number one, and $\ang{1, 1, 2, 2}\cong \mathbb{H}\mathbin{\perp}\mathbb{H}$ is primitively $2$-universal over $\z_p$ for any odd prime $p$.
Hence, we may assume that
\[
 \ell_2\cong\begin{cases}
  \ang{1, -1}\text{, }\ang{\epsilon, 4\delta}\text{, }\ang{\epsilon, 16\alpha}\text{, }\\
  \ang{2, -2}\text{, }\ang{2\epsilon, 8\alpha}\text{, }\\
  \ang{4\epsilon, 4\delta}\text{ with }\epsilon\delta\equiv -1\Mod4\text{, }\ang{4\epsilon, 16\alpha}\text{,}\\
  \mathbb{H}^2\text{,}\\
  \mathbb{H}\text{ or }\mathbb{A}\text{.}
 \end{cases}
\]
We define the binary $\z$-lattices
\[
 \ell'(u,t) \cong \left(\begin{smallmatrix}2a-3t^2 & 2ua + 2b\\2ua + 2b & 2u^2a + 4ub + 2c\end{smallmatrix}\right) \quad\text{and}\quad \ell''(u,t) \cong \left(\begin{smallmatrix}2a + 4ub + 2u^2c & 2b + 2uc\\2b + 2uc & 2c-3t^2\end{smallmatrix}\right)\text{.}
\]
Note that
\[
 \ell \cong \left(\begin{smallmatrix}a & ua + b\\ua + b & u^2a + 2ub + c\end{smallmatrix}\right) \cong \left(\begin{smallmatrix}a + 2ub + u^2c & b + uc\\b + uc & c\end{smallmatrix}\right)
\]
for any integer $u$. Hence, if $\ell'(u,t)$ or $\ell''(u,t)$ is primitively represented by $N\cong 2I_2\mathbin{\perp}\ang{4, 1}$ for some integers $u$ and $t$, then $\ell$ is primitively represented by $M$ by Lemma~\ref{lem:superd2h}. Since $N_p\cong \mathbb{H}\mathbin{\perp}\mathbb{H}$ for any odd prime $p$, $N_p$ is primitively $2$-universal over $\z_p$ for any odd prime $p$. Note that
\[
 \fs (\ell'(u,1)_2) = \fs (\ell''(u,1)_2) = \z_2\text{.}
\]

First, assume that $\fs(\ell_2) = \z_2$. Assume that $a$ is odd. Since $d\ell''(0,1)\equiv 2\Mod4$, $\ell''(0,1)_2$ is primitively represented by $N_2$. Note that $\ell''(0,1)$ is positive definite. Hence, $\ell''(0,1)$ is primitively represented by $N$. Now, assume that $a$ is even. Then $a\equiv 0\Mod4$ and $c$ is odd. In this case, $d\ell'(0,1)\equiv 2\Mod4$ and $\ell'(0,1)$ is positive definite. Hence, $\ell'(0,1)$ is primitively represented by $N$.

Now, assume that $\ell_2\cong\langle 2, -2\rangle$. Note that $a\equiv b\equiv c\equiv 0\Mod2$. First, suppose that $a\equiv 2\Mod4$. If $a \not\equiv -2\Mod{16}$, then
\[
 d\ell''(0,1)\equiv 4\Mod8 \quad \text{and} \quad d\ell''(0,1)\not\equiv -4\Mod{32}\text{.}
\]
Hence, $\ell''(0,1)$ is primitively represented by $N$. If $a\equiv -2\Mod{16}$, then
\[
 \fs (\ell''(0,2)_2) = 4\z_2 \quad \text{and} \quad d\ell''(0,2)\equiv 32\Mod{64}\text{.}
\]
Hence, $\ell''(0,2)$ is primitively represented by $N$. Now, suppose that $a\equiv 0\Mod4$. Then $a\equiv 0\Mod{16}$ and $c\equiv 2\Mod4$. If $c \not\equiv -2\Mod{16}$, then $\ell'(0,1)$ is primitively represented by $N$, and if $c\equiv -2\Mod{16}$, then $\ell'(0,2)$ is primitively represented by $N$.

Next, assume that $\ell_2\cong\langle 2\epsilon, 8\alpha\rangle$. Note that $a\equiv b\equiv c\equiv 0\Mod2$. If $a\equiv 2\Mod4$ and $a \not\equiv 6\Mod{16}$, then
\[
 d\ell''(0,1)\equiv 4\Mod8 \quad \text{and} \quad d\ell''(0,1)\not\equiv -4\Mod{32}\text{.}
\]
Hence, $\ell''(0,1)$ is primitively represented by $N$. Similarly, if $c\equiv 2\Mod4$ and $c \not\equiv 6\Mod{16}$, then $\ell'(0,1)$ is primitively represented by $N$. Now, assume that $a\equiv 6\Mod{16}$ or $c\equiv 6\Mod{16}$. Since we are assuming that
\[
 \ell_2\not\cong\langle 6, 16\alpha\rangle\text{,}
\]
we have $d\ell\equiv 16\Mod{32}$. Assume that $a\equiv c\equiv 6\Mod{16}$. Then, $b\equiv 2\Mod4$. Hence, there is an $\eta\in\{1, -1\}$ such that
\[
 a + 2\eta b + c \equiv 6\Mod8\text{.}
\]
Since $d\ell''(\eta, 1)\equiv 8\Mod{16}$, we have
\[
 \ell''(\eta, 1)_2\cong\langle\epsilon, 8\delta\rangle\text{,}
\]
which is primitively represented by $N_2 \cong \langle 1, 2, 2, 4 \rangle$. Furthermore, since $\ell''(\eta, 1)$ is positive definite, it is primitively represented by $N$. Next, assume that $a\equiv 6\Mod{16}$ and $c\not\equiv 6\Mod{16}$. Then either $c\equiv 8\Mod{16}$ and $b\equiv 0\Mod8$, or $c\equiv 0\Mod{16}$ and $b\equiv 4\Mod8$. In any case,
\[
 a - 2b + c\equiv -2\Mod{16}\text{.}
\]
Since $d\ell''(-1,1)\equiv 12\Mod{32}$, we have
\[
 \ell''(-1, 1)_2\cong\langle-3, -4\rangle\text{,}
\]
which is primitively represented by $N_2 \cong \langle 1, 2, 2, 4 \rangle$. Furthermore, since $\ell''(-1, 1)$ is positive definite, it is primitively represented by $N$. Finally, assume that $a\not\equiv 6\Mod{16}$ and $c\equiv 6\Mod{16}$. Then, similarly to the above, $\ell'(-1, 1)$ is primitively represented by $N$ in this case.

Now, assume that $\fs(\ell_2) = 4\z_2$. Note that $a\equiv b\equiv c\equiv 0\Mod4$. If $a\equiv 4\Mod8$, then $d\ell''(0,1)\equiv 8\Mod{16}$. Since $\ell''(0,1)$ is positive definite, $\ell''(0,1)$ is primitively represented by $N$. Similarly, if $c\equiv 4\Mod8$, then $\ell'(0,1)$ is primitively represented by $N$ in this case.

Next, assume that $\ell_2 \cong \mathbb{H}^2$. Note that $a\equiv b-2\equiv c\equiv 0\Mod4$. Assume that $a\equiv 4\Mod8$. Since $d\ell''(0,1)\equiv 8\Mod{16}$, we have
\[
 \ell''(0, 1)_2\cong\langle\epsilon, 8\delta\rangle\text{,}
\]
which is primitively represented by $N_2 \cong \langle 1, 2, 2, 4 \rangle$. Since $\ell''(0,1)$ is positive definite, $\ell''(0,1)$ is primitively represented by $N$. Next, assume that $c\equiv 4\Mod8$. Then, similarly to the above, $\ell'(0,1)$ is primitively represented by $N$ in this case. Finally, assume that $a\equiv c\equiv 0\Mod8$. Since
\[
 a - 2b + c\equiv 4\Mod8\text{,}
\]
we have $d\ell''(-1, 1)\equiv 8\Mod{16}$. Since $\ell''(-1,1)$ is positive definite, $\ell''(-1,1)$ is primitively represented by $N$.

Finally, assume that $\ell_2 \cong \mathbb{H}$ or $\mathbb{A}$. Note that $a\equiv b-1\equiv c\equiv 0\Mod2$. Assume that $a\equiv 6\Mod8$. Since $d\ell''(0,1)\equiv 8\Mod{16}$, we have
\[
 \ell''(0, 1)_2\cong\langle\epsilon, 8\delta\rangle\text{,}
\]
which is primitively represented by $N_2 \cong \langle 1, 2, 2, 4 \rangle$. Hence, $\ell''(0,1)$ is primitively represented by $N$. Assume that $c\equiv 6\Mod8$. Then, similarly to the above, $\ell'(0,1)$ is primitively represented by $N$ in this case. Now, suppose that neither $a$ nor $c$ is congruent to $6$ modulo $8$. Then we have
\[
 a \equiv 2\Mod8 \quad \text{or} \quad a \equiv 0\Mod 4\text{,}
\]
and the same with $c$. First, assume that $a\equiv c\equiv 2\Mod8$. Then, there is an $\eta\in\{1, -1\}$ such that
\[
 a + 2\eta b + c \equiv 6\Mod8\text{.}
\]
Since $d\ell''(\eta, 1)\equiv 8\Mod{16}$, $\ell''(\eta, 1)$ is primitively represented by $N_2 \cong \langle 1, 2, 2, 4 \rangle$. Hence, $\ell''(\eta, 1)$ is primitively represented by $N$ if it is positive definite, that is, if $a\ge 7$, or if $a = 2$ and $c\ge 15$. Clearly, $\left(\begin{smallmatrix}2&1\\1&2\end{smallmatrix}\right)$ and $\left(\begin{smallmatrix}2&1\\1&10\end{smallmatrix}\right)$ are primitively represented by $M$. Next, assume that $a\equiv 0\Mod4$ and $c\equiv 2\Mod8$. Then
\[
 4a - 4b + c\equiv 6\Mod8\text{.}
\]
Since $d\ell'(-2, 1)\equiv 8\Mod{16}$, we have $\ell'(-2, 1)_2 \cong\langle\epsilon, 8\delta\rangle$, which is primitively represented by $N_2$. Since
\[
 \ell'(-2, 1) \cong \left(\begin{smallmatrix}2a-3 & -4a + 2b\\-4a + 2b & 8a - 8b + 2c\end{smallmatrix}\right)
\]
is positive definite, it is primitively represented by $N$. Now, assume that $a\equiv 2\Mod8$ and $c\equiv 0\Mod4$. Then, similarly to the above, $\ell''(-2, 1)$ is primitively represented by $N$ if it is positive definite, that is, if $a\ge 11$, or if $a = 10$ and $c\ge 4$. The case when $a = 2$ will be postponed to the end of this proof. Finally, assume that $a\equiv c\equiv 0\Mod4$. Then, there is an $\eta\in\{1, -1\}$ such that
\[
 a + 2\eta b + c \equiv 6\Mod8\text{.}
\]
Hence, $\ell''(\eta, 1)$ is primitively represented by $N$ if it is positive definite, that is, if $a\ge 7$, or if $a = 4$ and $c\ge 5$. Clearly $\left(\begin{smallmatrix}4&1\\1&4\end{smallmatrix}\right)$ is primitively represented by $M$.

Now, suppose that $a=2$, $b=1$, and $c\equiv 0\Mod4$. If $c\not\equiv 0\Mod{16}$, then $\ang{1,1,2}$ represents $c-2$. If we choose a vector $v$ in the $3$-section of $M$ such that $Q(v) = c-2$, then clearly, $\z[e_4, v + e_5]$ is a primitive sublattice of $M$ isometric to $\left(\begin{smallmatrix}2&1\\1&c\end{smallmatrix}\right)$. If $c\equiv 0\Mod{16}$, then $\ang{1,1,2}$ primitively represents $c-14$. If we choose a vector $w$ in the $3$-section of $M$ such that $Q(v) = c-14$, then clearly,
\[
 \z[e_4, w - 2e_4 + 3e_5]
\]
is a primitive sublattice of $M$ isometric to $\left(\begin{smallmatrix}2&1\\1&c\end{smallmatrix}\right)$.
\end{proof}

For type H\textsuperscript{i} lattices, the $5$-section splits $L$. Hence, one may prove the primitive $2$-universalities of type H\textsuperscript{i} lattices in the similar way to Theorem~\ref{thm:pf42} by using the above lemma. Note that for the proofs of primitive $2$-universalities of type H\textsuperscript{ii}, H\textsuperscript{iii}, and H\textsuperscript{iv} lattices, we need Lemma~\ref{lem:superhn} given below as well as Lemma~\ref{lem:h}. In particular, the proof of the primitive $2$-universality of the lattice H$^{\mbox{\scriptsize iii}}_6$ will be provided in Subsection~\ref{ssec:55} since we need Lemma~\ref{lem:notog} additionally. The proof of the primitive $2$-universality of the lattice H$^{\mbox{\scriptsize ii}}_5$ will be provided in Theorem~\ref{thm:pfh25}.

\begin{lem}\label{lem:superhn}
Let $\ell \cong \left(\begin{smallmatrix}a&b\\b&c\end{smallmatrix}\right)$ be a binary $\z$-lattice and let $q$ and $r$ be positive integers.
\begin{enumerate}
\item Suppose that $\ang{1,1,2,2}$ primitively represents $\left(\begin{smallmatrix}2a-A & 2b-B\\2b-B & 2c-C\end{smallmatrix}\right)$, where
\[
 \begin{pmatrix}A&B\\B&C\end{pmatrix} = \begin{pmatrix}s_1 & s_2\\t_1 & t_2\end{pmatrix} \begin{pmatrix}2q-1&0\\0&2r-1\end{pmatrix} \begin{pmatrix}s_1 & t_1\\s_2 & t_2\end{pmatrix}
\]
for some integers $s_1$, $s_2$, $t_1$, and $t_2$ such that at least one of $A$ and $C$ is odd. Then $\ell$ is primitively represented by $I_2\mathbin{\perp}\left(\begin{smallmatrix}2&1\\1&q\end{smallmatrix}\right)\mathbin{\perp}\left(\begin{smallmatrix}2&1\\1&r\end{smallmatrix}\right)$.
\item Suppose that $\ang{1,2,2,4}$ primitively represents $\left(\begin{smallmatrix}2a-A & 2b-B\\2b-B & 2c-C\end{smallmatrix}\right)$, where
\[
 \begin{pmatrix}A&B\\B&C\end{pmatrix} = \begin{pmatrix}s_1 & s_2\\t_1 & t_2\end{pmatrix} \begin{pmatrix}2q-1&1\\1&2r-1\end{pmatrix} \begin{pmatrix}s_1 & t_1\\s_2 & t_2\end{pmatrix}
\]
for some integers $s_1$, $s_2$, $t_1$, and $t_2$. Then $\ell$ is primitively represented by $I_2\mathbin{\perp}\ang{2}\mathbin{\perp}\left(\begin{smallmatrix}2&1&1\\1&q&1\\1&1&r\end{smallmatrix}\right)$.
\item Suppose that $\ang{1,1,2,2}$ primitively represents $\left(\begin{smallmatrix}2a-A & 2b-B\\2b-B & 2c-C\end{smallmatrix}\right)$, where
\[
 \begin{pmatrix}A&B\\B&C\end{pmatrix} = \begin{pmatrix}s_1 & s_2\\t_1 & t_2\end{pmatrix} \begin{pmatrix}2q-1&1\\1&2r-2\end{pmatrix} \begin{pmatrix}s_1 & t_1\\s_2 & t_2\end{pmatrix}
\]
for some integers $s_1$, $s_2$, $t_1$, and $t_2$ such that at least one of $A$ and $C$ is odd. Then $\ell$ is primitively represented by $I_2\mathbin{\perp}\left(\begin{smallmatrix}2&0&0&1\\0&2&1&1\\0&1&q&1\\1&1&1&r\end{smallmatrix}\right)$.
\end{enumerate}
\end{lem}

\begin{proof}
Since the proof is quite similar to that of Lemma~\ref{lem:357super}, it is left to the readers.
\end{proof}

\begin{thm}\label{thm:pfh25}
The $\z$-lattice \textup{H$^{\mbox{\scriptsize ii}}_5$} $\cong I_2\mathbin{\perp}\mathbb{A}\mathbin{\perp}\left(\begin{smallmatrix}2&1\\1&5\end{smallmatrix}\right)$ is primitively $2$-universal.
\end{thm}

\begin{proof}
Let $\ell\cong \left(\begin{smallmatrix}a&b\\b&c\end{smallmatrix}\right)$ ($0\le 2b\le a\le c$) be a $\z$-lattice which is not primitively represented by the $5$-section of $L$. Then, by Lemma~\ref{lem:h}, we may assume that $\ell$ satisfies either
\begin{itemize}
\item[(i)] $a$ or $c\equiv 6\Mod{16}$ and $d\equiv 0\Mod{32}$, or
\item[(ii)] $a\equiv c\equiv 0\Mod8$ and $b\equiv 0\Mod4$.
\end{itemize}
According to Lemma \ref{lem:superhn}, if a $\z$-lattice
\[
 \ell' \cong \begin{pmatrix}2a-3&2b\\2b&2c-(2q-1)\end{pmatrix}\text{,}
\]
is primitively represented by $N\cong \ang{1,1,2,2}$, then $\ell$ is primitively represented by $L$. Furthermore, according to Lemma~\ref{lem:superd2h}, if
\[
 \ell'' \cong \begin{pmatrix}2a-(2q-1)&2b\\2b&2c\end{pmatrix}\quad\text{or}\quad \ell''' \cong \begin{pmatrix}2a&2b\\2b&2c-(2q-1)\end{pmatrix}
\]
is primitively represented by $N'\cong \ang{1,2,2,4}$, then $\ell$ is primitively represented by $I_2\mathbin{\perp}\ang{2}\mathbin{\perp}\left(\begin{smallmatrix}2&1\\1&5\end{smallmatrix}\right)$, and hence by $L$.

Assume that case (i) holds. Note that $a\equiv b\equiv c\equiv 0\Mod2$ in this case. If $a\equiv c\equiv 2\Mod4$, then $d\ell'\equiv 3\Mod8$. Thus $\ell'_2$ is primitively represented by $N_2$ in this case. Since $\ell'$ is positive definite, it is primitively represented by $N$. If $a-2\equiv c\equiv 0\Mod4$, then $\fs (\ell'''_2) = \z_2$ and $d\ell'''\equiv 20\Mod{32}$. Thus, one may easily show that $\ell'''_2$ is primitively represented by $N'_2$ in this case. Since $\ell'''$ is positive definite, it is primitively represented by $N'$. Similarly, if $a\equiv c-2\equiv 0\Mod4$, then $\ell''$ is primitively represented by $N'$.

Now, assume that case (ii) holds. Clearly, we have $d\ell'\equiv 3\Mod8$. Since $\ell'$ is positive definite, it is primitively represented by $N$.
\end{proof}

\subsection{Case 5: $L$ has a quaternary primitive sublattice of class number one or two}
\label{ssec:54}

In this subsection, we prove the primitive $2$-universalities of the lattices
\[
 \text{B$^{\mbox{\scriptsize ii}}_5$, B$^{\mbox{\scriptsize ii}}_6$, D$^{\mbox{\scriptsize iii}}_5$, I$^{\mbox{\scriptsize ii}}_4$, I$^{\mbox{\scriptsize ii}}_5$, I$^{\mbox{\scriptsize iii}}_4$, and J$_3$.}
\]
Note that the $5$-section of $L$ has class number one and does not split $L$ orthogonally. In these cases, we make use of some information on the primitive binary exceptions of quaternary primitive sublattices. The method of the proof depends heavily on whether the discriminant of such a quaternary sublattice which we are considering is a perfect square or not.

First, assume $L\cong$ B$^{\mbox{\scriptsize ii}}_5$ or B$^{\mbox{\scriptsize ii}}_6$. The $4$-section $I_4$ is a quaternary orthogonal summand of class number one whose discriminant is a perfect square. Hence, the primitive $2$-universalities could be proved in a similar manner to that of Theorem~\ref{thm:pf42}.

Recall that a finite sequence of vectors $v_1, \dots, v_m$ in $\z^n$ ($m\le n$) is primitive if and only if the greatest common divisor $g$ of the determinants of all $m\times m$ submatrices of the coefficient matrix of $v_1, \dots, v_m$, which is defined by the $m\times n$ matrix whose rows are $v_1, \dots, v_m$, is $\pm 1$. Also, we say that $v_1, \dots, v_m$ is $p$-primitive for a prime $p$ if $g$ is prime to $p$. Then it is clear that $v_1, \dots, v_m$ is primitive if and only if it is $p$-primitive for any prime $p$.

\begin{lem}\label{lem:notog}
Let $L = \z e_1 + \cdots + \z e_{n+1}$ be a free $\z$-module of rank $n+1$, and let $M = \z e_1 + \cdots + \z e_n$. Suppose that $v_1, \dots, v_m$ are vectors in $M$ for some $1 \le m \le n$.
\begin{enumerate}
\item Suppose that $v_1, \dots, v_m$ is $p$-primitive for some prime $p$. Then $v_1, \dots,\allowbreak v_{m-1}, v_m+pw$ is also $p$-primitive for any $w\in M$.
\item Suppose that $v_1, \dots, v_m$ is $p$-primitive for some odd prime $p$. Then for any $w\in M$, either $v_1, \dots, v_{m-1}, v_m+w$ or $v_1, \dots, v_{m-1}, v_m-w$ is also $p$-primitive.
\item Suppose that $v_1, \dots, v_m$ is primitive. For a vector $y = y_1 e_1 + \cdots + y_n e_n + y_{n+1} e_{n+1}\in L$, put
\begin{align*}
 \mathcal{P}(y) & = \{p : p\text{ is a prime that divides }\gcd(y_1, \dots, y_n, y_{n+1})\}\text{,}\\
 \mathcal{P}(y_{n+1}) & = \{p : p\text{ is a prime that divides }y_{n+1}\}\text{.}
\end{align*}
If $\mathcal{P}(y_{n+1})\setminus\mathcal{P}(y) = \varnothing$, then $v_1, \dots, v_{m-1}, v_m+y$ is primitive. If $\mathcal{P}(y_{n+1})\setminus\mathcal{P}(y) = \{p\}$ for some odd prime $p$, then either $v_1, \dots, v_{m-1},\allowbreak v_m+y$ or $v_1, \dots, v_{m-1}, v_m-y$ is primitive.
\end{enumerate}
\end{lem}

\begin{proof}
(1) The lemma follows from the fact that the determinant of any $m\times m$ submatrix of the $m\times n$ coefficient matrix of $v_1, \dots, v_m$ is congruent to the determinant of the corresponding $m\times m$ submatrix of the $m\times n$ coefficient matrix of $v_1, \dots, v_{m-1}, v_m+pw$ modulo $p$.

\noindent (2) Suppose on the contrary that both
\[
 v_1, \dots, v_{m-1}, v_m+w \quad\text{and}\quad v_1, \dots, v_{m-1}, v_m-w
\]
are not $p$-primitive. This implies that the determinant of any $m\times m$ submatrix of the $m\times n$ coefficient matrices $C(\eta)$ of $v_1, \dots, v_{m-1}, v_m + \eta w$ is a multiple of $p$ for any $\eta\in\{1, -1\}$. Observe that, by multilinearity of the determinant, the determinant of any $m\times m$ submatrix of $C(\eta)$ is equal to
\begin{multline*}
 \det(\text{the corresponding $m\times m$ submatrix of $C$})\\
 + \eta\cdot\det(\text{the corresponding $m\times m$ submatrix of $C'$})\text{,}
\end{multline*}
where $C$ is the $m\times n$ coefficient matrix of $v_1, \dots, v_m$, and $C'$ is that of $v_1, \dots, v_{m-1}, w$. Since $p$ is odd, if the determinants of any $m\times m$ submatrix of $C(1)$ and the corresponding $m\times m$ submatrix of $C(-1)$ are multiples of $p$ simultaneously, then so are the determinants of the corresponding $m\times m$ submatrices of $C$ and $C'$. This implies that $v_1, \dots, v_m$ is not $p$-primitive, which is a contradiction.

\noindent (3) We have to show that the greatest common divisor of the determinants of $m\times m$ submatrices of $m\times (n+1)$ coefficient matrix of $v_1, \dots, v_{m-1}, v_m + \eta y$ is $1$ for some $\eta\in\{1, -1\}$. Let $g_1$ be the greatest common divisor of the determinants of all $m\times m$ submatrices containing the $(n+1)$-th column, and $g_2$ be the greatest common divisor of the determinants of those not containing the column. Then it suffices to show that $(g_1, g_2) = 1$. Let
\[
 w = y_1 e_1 + \cdots + y_n e_n\in M\text{.}
\]
Then $g_2$ is equal to the greatest common divisor of the determinants of all $m\times m$ submatrices of $m\times n$ coefficient matrix of $v_1, \dots, v_{m-1}, v_m + \eta w$. Note that $g_1 = |y_{n+1}|$. Hence it suffices to show that $v_1, \dots, v_{m-1}, v_m + \eta w$ is $q$-primitive for any $q\in \mathcal{P}(y_{n+1})$. If $\mathcal{P}(y_{n+1})\setminus\mathcal{P}(y) = \varnothing$, it follows directly from (1). If $\mathcal{P}(y_{n+1})\setminus\mathcal{P}(y) = \{p\}$, it follows from (1) that both $v_1, \dots, v_{k-1}, v_k + w$ and $v_1, \dots, v_{k-1}, v_k - w$ are $q$-primitive for any $q\in\mathcal{P}(y_{n+1})$ such that $q\ne p$ (equivalently, for any $q\in \mathcal{P}(y)$), and it follows from (2) that at least one of the two is $p$-primitive.
\end{proof}

Now, assume that $L\cong$ J$_3$. Note that the quaternary sublattice $\langle 1, 2\rangle\mathbin{\perp}\left(\begin{smallmatrix}3&1\\1&3\end{smallmatrix}\right)$ of $L$ with a perfect square discriminant has class number one. Below we provide a complete proof of the primitive $2$-universality of J$_3$, which is one of the most complicated cases among all candidates.

\begin{thm}\label{thm:pfj}
\noindent\textup{(1)} Let $\ell\cong \left(\begin{smallmatrix}a&b\\b&c\end{smallmatrix}\right)$ be a binary $\z$-lattice such that $a\equiv 1\Mod2$ and $c\equiv 0\Mod{16}$. If $\left(\begin{smallmatrix}a&b\\b&c-6\end{smallmatrix}\right)$ is primitively represented by the $\z$-lattice $\ang{1, 2}\mathbin{\perp}\left(\begin{smallmatrix}3&1\\1&3\end{smallmatrix}\right)$, then $\ell$ is primitively represented by the $\z$-lattice $\ang{1}\mathbin{\perp}\mathbb{A}\mathbin{\perp}\left(\begin{smallmatrix}3&1\\1&3\end{smallmatrix}\right)$.

\noindent\textup{(2)} The $\z$-lattice $L\cong$ \textup{J$_3$} $\cong I_2\mathbin{\perp}\mathbb{A}\mathbin{\perp}\left(\begin{smallmatrix}3&1\\1&3\end{smallmatrix}\right)$ is primitively $2$-universal.
\end{thm}

\begin{proof}
(1) We fix bases for $\z$-lattices $\ang{1, 2}\mathbin{\perp}\left(\begin{smallmatrix}3&1\\1&3\end{smallmatrix}\right)$ and $\ang{1}\mathbin{\perp}\mathbb{A}\mathbin{\perp}\left(\begin{smallmatrix}3&1\\1&3\end{smallmatrix}\right)$ corresponding to the given Gram matrices. With respect to such bases, if there exists a primitive sublattice $\z\left[\begin{smallmatrix}c_1 & c_2 & c_3 & c_4\\d_1 & d_2 & d_3 & d_4\end{smallmatrix}\right]\cong \left(\begin{smallmatrix}a&b\\b&c-6\end{smallmatrix}\right)$ of $\ang{1, 2}\mathbin{\perp}\left(\begin{smallmatrix}3&1\\1&3\end{smallmatrix}\right)$ for $c_i$, $d_i\in\z$, then clearly, the sublattice
\[
 \z\begin{bmatrix}c_1 & c_2 & 0 & c_3 & c_4\\d_1 & d_2-1 & 2 & d_3 & d_4\end{bmatrix}
\]
of $\ang{1}\mathbin{\perp}\mathbb{A}\mathbin{\perp}\left(\begin{smallmatrix}3&1\\1&3\end{smallmatrix}\right)$ is isometric to $\ell$. Since the greatest common divisor of the determinants of all $2\times 2$ submatrices containing the third column is $2$, it suffices to show that $\left(\begin{smallmatrix}c_1 & c_3 & c_4\\d_1 & d_3 & d_4\end{smallmatrix}\right)$ is $2$-primitive. Since $a$ is odd, exactly one of $c_1$, $c_3$, $c_4$ is odd or all of them are odd. Furthermore, since
\[
 d_1^2 + 2d_2^2 + 3d_3^2 + 2d_3 d_4 + 3d_4^2 \equiv 10\Mod{16}\text{,}
\]
$d_1$ is even and $d_3$, $d_4$ are odd. Therefore, $\left(\begin{smallmatrix}c_1 & c_3 & c_4\\d_1 & d_3 & d_4\end{smallmatrix}\right)$ is $2$-primitive.

\noindent (2) Denote by $M$ the $5$-section of $L$. Let $M' = \z[e_1, e_3, e_4, e_5, e_6]$ and $N = \z[e_1, e_3, e_5, e_6]$ be primitive sublattices of $L$. Note that
\[
 M' \cong \ang{1}\mathbin{\perp}\begin{pmatrix}2&1\\1&2\end{pmatrix}\mathbin{\perp}\begin{pmatrix}3&1\\1&3\end{pmatrix}\quad \text{and} \quad N \cong \ang{1,2}\mathbin{\perp}\begin{pmatrix}3&1\\1&3\end{pmatrix}\text{.}
\]
Let $\ell\cong \left(\begin{smallmatrix}a&b\\b&c\end{smallmatrix}\right)$ ($0\le 2b\le a\le c$) be a binary $\z$-lattice which is primitively represented by neither $M$ nor $N$. Then $\ell$ satisfies either
\begin{itemize}
\item[(i)] $a$ or $c\equiv 1\Mod8$ and $d\ell\equiv 0\Mod{16}$, or
\item[(ii)] $a\equiv c\equiv 0\Mod4$ and $b$ is even.
\end{itemize}
If $c\le 32$, then one may directly check that $\ell$ is primitively represented by $L$. Hence, we may assume that $c\ge 33$. Also, we assume that $a\ne 4$ for the moment. Consider the $\z$-lattices
\begin{align*}
 \ell' = \ell'_u(s,t;s') & \cong \begin{pmatrix}a + 2ub + u^2c - s^2 - 6s'^2 & b + uc - st\\b + uc - st & c - t^2\end{pmatrix}\text{,}\\
 \ell'' = \ell''_u(s,t;t') & \cong \begin{pmatrix}a - s^2 & ua + b - st\\ua + b - st & u^2a + 2ub + c - t^2 - 6t'^2\end{pmatrix}\text{,}
\end{align*}
where $u$, $s$, $t$, $s'$ and $t'$ are integers. Observe that the orthogonal complement of $N$ in $M'$ is $N^\perp = \z[- e_3 + 2e_4]\cong\ang{6}$. Hence, by Lemma~\ref{lem:notog}, if $\ell'_u(s, t; s')$ ($\ell''_u(s, t; t')$) is primitively represented by $N$ for some $s'$ ($t'$, respectively) even, then $\ell$ is primitively represented by $L$. Moreover, by (1) given above, if all of the following three conditions hold, then $\ell$ is primitively represented by $L$:
\begin{itemize}
\item[(a)] $\ell'_u(s, t; 1)$ ($\ell''_u(s, t; 1)$) is primitively represented by $N$;
\item[(b)] $a + 2ub + u^2c - s^2$ ($u^2a + 2ub + c - t^2$) $\equiv 0\Mod{16}$;
\item[(c)] $c - t^2$ ($a - s^2$, respectively) is odd.
\end{itemize}
Assume that case (i) holds. First, suppose that $a$ is odd. We consider the $\z$-lattice
\[
 \ell''_u(0,0;1)\cong\bigl(\begin{smallmatrix}a & ua + b\\ua + b & u^2a + 2ub + c - 6\end{smallmatrix}\bigr)\text{.}
\]
Since $d\ell''_u(0,0;1)\equiv 2\Mod4$, $\ell''_u(0,0;1)_2$ is primitively represented by $N_2\cong \ang{1,2}\mathbin{\perp}\left(\begin{smallmatrix}3&1\\1&3\end{smallmatrix}\right)$. Hence, for any integer $u$, $\ell''_u(0,0;1)$ is primitively represented by $N$. Since $a$ is odd, there is a $u\in\{-2, -1, 0, 1\}$ such that $ua + b\equiv 0\Mod4$, which implies that $u^2a + 2ub + c\equiv 0\Mod{16}$. Hence, we are done by (1). Now, suppose that $a$ is even. Then $a\equiv 0$, $4\Mod{16}$ and $c$ is odd. Hence, similarly to the above, if we take an integer $u\in\{-2, -1, 0, 1\}$ such that $u^2a + 2ub + c\equiv 0\Mod{16}$, then $\ell'_u(0,0;1)$ is primitively represented by $N$.

Now, assume that case (ii) holds. First, suppose that either $a\equiv 0$, $4\Mod{16}$ or $c\equiv 0$, $4\Mod{16}$. If we define a $\z$-lattice $\ell'''$ by
\[
 \ell''' = \begin{cases}
  \ell''_0(1,0;1) & \text{if }a\equiv 0\Mod{16}\text{,}\\
  \ell''_0(1,2;1) & \text{if }a\equiv 4\Mod{16}\text{,}\\
  \ell'_0(0,1;1) & \text{if }c\equiv 0\Mod{16}\text{,}\\
  \ell'_0(2,1;1) & \text{if }c\equiv 4\Mod{16}\text{,}
 \end{cases}
\]
then $d\ell'''\equiv 2\Mod4$. Hence, $\ell'''$ is primitively represented by $N$. Therefore, by (1), $\ell$ is primitively represented by $L$ in each case. Now, suppose that $a$, $c\equiv 8$ or $12\Mod{16}$. First, assume that $b\equiv 2\Mod4$. One may easily show that there is an $\eta\in\{1, -1\}$ such that
\[
 a + 2\eta b + c \equiv 0\text{ or }4\Mod{16}\text{.}
\]
Hence, one of $\ell''_\eta(1,0;1)$ or $\ell''_\eta(1,2;1)$ is primitively represented by $N$. Therefore, by (1), $\ell$ is primitively represented by $L$. Now, assume that $b\equiv 0\Mod4$. If we define a $\z$-lattice $\ell^{(4)}$ by
\[
 \ell^{(4)} = \begin{cases}
  \ell''_0(0,1;0) & \text{if }a\equiv 8\Mod{16}\text{,}\\
  \ell'_0(1,0;0) & \text{if }c\equiv 8\Mod{16}\text{,}\\
  \ell''_0(0,4;0) & \text{if }a\equiv c\equiv 12\Mod{16}\text{ and }b\equiv 0\Mod8\text{,}\\
  \ell''_0(0,0;2) & \text{if }a\equiv c\equiv 12\Mod{16}\text{ and }b\equiv 4\Mod8\text{,}
 \end{cases}
\]
then one may easily show that
\[
 d\ell^{(4)}\equiv 8\Mod{16}\text{,}\quad \ell^{(4)}_2\cong\langle 12, -4\rangle\text{,}\quad \text{or}\quad d\ell^{(4)}\cong 32\Mod{64}\text{.}
\]
Hence, $\ell^{(4)}$ is primitively represented by $N$, which implies that $\ell$ is primitively represented by $L$ in each case.

Finally, we consider the remaining case when $a = 4$. Note that $b = 0$ or $b = 2$. It is well known that the $4$-section $N'\cong I_2\mathbin{\perp}\mathbb{A}$ of $L$ is primitively $1$-universal (see \cite{Bu10}). If we choose a primitive vector $v$ in $N'$ such that $Q(v) = c$, then clearly, $\z[e_5 - e_6, v]$ is a primitive sublattice of $L$ which is isometric to $\ang{4, c}$. If we choose a vector $w$ in $N'$ such that $Q(w) = c - 3$, then clearly,
\[
 \z[e_5 - e_6, w + e_5]
\]
is a primitive sublattice of $L$ which is isometric to $\left(\begin{smallmatrix}4&2\\2&c\end{smallmatrix}\right)$.
\end{proof}

Next, we consider the case when $L\cong$ D$^{\mbox{\scriptsize iii}}_5$ or I$^{\mbox{\scriptsize ii}}_5$. The quaternary primitive $\z$-sublattice $N$ of $L$ given in Table~\ref{tblr:classtwo} has class number two and its genus mate $N'$ is also primitively represented by $L$. Hence, any binary $\z$-lattice that is represented by the genus of $N$ is primitively represented by $L$. Note that $dN = 16$ and the sublattice $\z[e_3, e_4, e_5, e_6]$ of $L$ is isometric to $N$ for both cases.

\begin{table}[ht]
\caption{The $\z$-lattice $N$ and its genus mate $N'$}
\label{tblr:classtwo}
\renewcommand{\arraystretch}{1.2}
\begin{tabular}{c|c|c}
 \hline
 $L$ & $N$ & $N'$\\\hline
 D$^{\mbox{\scriptsize iii}}_5$ & $\langle 1\rangle \mathbin{\perp}\left(\begin{smallmatrix}2&0&1\\0&2&1\\1&1&5\end{smallmatrix}\right)$\rule[-9pt]{0pt}{24pt} & $I_3\mathbin{\perp}\langle 16\rangle$\\\hline
 I$^{\mbox{\scriptsize ii}}_5$ & $\left(\begin{smallmatrix}2&1&1&1\\1&2&1&1\\1&1&2&0\\1&1&0&5\end{smallmatrix}\right)$\rule[-12pt]{0pt}{30pt} & $I_2\mathbin{\perp}\left(\begin{smallmatrix}4&2\\2&5\end{smallmatrix}\right)$\\\hline
\end{tabular}
\end{table}

Below we provide a complete proof of the primitive $2$-universality of D$^{\mbox{\scriptsize iii}}_5$. The proof of the primitive $2$-universality of I$^{\mbox{\scriptsize ii}}_5$ is quite similar to this. Recall that $\alpha$, $\beta$ denote integers in $\z_p$, and $\epsilon$, $\delta$ denote units in $\z_p$, unless stated otherwise, where the prime $p$ could be easily verified from the context.

\begin{thm}\label{thm:pfd35}
\noindent\textup{(1)} If a quaternary $\z_2$-lattice $I_3\mathbin{\perp}\ang{16}$ does not primitively represent a binary $\z_2$-lattice $\ell_2$, then $\ell_2$ satisfies one of the following conditions:

 $\ell_2\cong\ang{-1, \alpha}$, $\ang{1, 12}$, $\ang{1, 8}$, $\ang{1, 40}$, $\ang{5, 24}$, $\ang{5, -8}$, $\ang{1, 64\alpha}$, $\ang{3, 32\alpha}$, $\ang{5, 64\alpha}$, $\ang{2, 10}$, $\ang{6, 6}$, $\ang{2\epsilon, 12}$, $\ang{2, 8}$, $\ang{-2, 24}$, $\ang{2\epsilon, 32\delta}$ with $\epsilon\delta\equiv 3\Mod8$, $\ang{2\epsilon, 128\alpha}$, $\fn (\ell_2)\subseteq 4\z_2$, or $\q_2 \ell_2\cong \q_2 \mathbb{H}$.
 
\noindent\textup{(2)} Let $\ell \cong \left(\begin{smallmatrix}a&b\\b&c\end{smallmatrix}\right)$ be a binary $\z$-lattice. Suppose that, for some integers $s$, $t$, and $t'$,
\[
 \ell'\cong \begin{pmatrix}a-2s^2 & b-2st\\b-2st & c-2t^2 - 8t'^2\end{pmatrix}
\]
is positive definite and $\ell'_2$ is primitively represented by the quaternary $\z_2$-lattice $I_3\mathbin{\perp}\ang{16}$. Then $\ell$ is primitively represented by the $\z$-lattice $L\cong$ \textup{D$^{\mbox{\scriptsize iii}}_5$}.

\noindent\textup{(3)} The $\z$-lattice $L\cong$ \textup{D$^{\mbox{\scriptsize iii}}_5$} $\cong I_3\mathbin{\perp}\left(\begin{smallmatrix}2&0&1\\0&2&1\\1&1&5\end{smallmatrix}\right)$ is primitively $2$-universal.
\end{thm}

\begin{proof}
(1) One may easily verify the assertion by a direct computation.

\noindent (2) Consider two primitive sublattices of $L$,
\[
 N = \z[e_3, e_4, e_5, e_6]\quad \text{and}\quad N' = \z[e_1, e_2, e_3, e_4 + e_5 - 2e_6]\cong I_3\mathbin{\perp}\ang{16}\text{,}
\]
where the ordered basis $\{e_i\}_{i=1}^6$ for $L$ corresponds to the Gram matrix given in the statement. Moreover, we fix an ordered basis for $N'$ corresponding to the Gram matrix given above. Note that the class number of $N$ is two and $N'$ is the other class is the genus of $N$. Hence, if $\ell'$ satisfies all conditions given above, then $\ell'$ is primitively represented by either $N$ or $N'$.

If $\ell'$ is primitively represented by $N$, then one may directly check that $\ell$ is primitively represented by $L$. Now, suppose that the primitive sublattice
\[
 \z\begin{bmatrix}c_1 & c_2 & c_3 & c_4\\d_1 & d_2 & d_3 & d_4\end{bmatrix}
\]
of $N'$ is isometric to $\ell'$. Then clearly, the sublattice
\[
 \z\begin{bmatrix}c_1 & c_2 & c_3 & c_4 + s & c_4 & -2c_4\\d_1 & d_2 & d_3 & d_4 + t & d_4 + 2t' & -2d_4\end{bmatrix}
\]
of $L$ is isometric to $\ell$. To see that such a sublattice is primitive, note that the greatest common divisor of all $2\times 2$ submatrices of the above coefficient matrix divides $(g_1, g_2)$, where $g_1$ and $g_2$ are the greatest common divisor of all $2\times 2$ submatrices of
\[
 \begin{pmatrix}c_1 & c_2 & c_3 & -2c_4\\d_1 & d_2 & d_3 & -2d_4\end{pmatrix}\quad \text{and}\quad \begin{pmatrix}c_1 & c_2 & c_3 & c_4\\d_1 & d_2 & d_3 & d_4 + 2t'\end{pmatrix}\text{,}
\]
respectively. Since $\left(\begin{smallmatrix}c_1 & c_2 & c_3 & c_4\\d_1 & d_2 & d_3 & d_4\end{smallmatrix}\right)$ is primitive, $g_1$ divides $2$ and $g_2$ is odd. Hence, $(g_1, g_2) = 1$.

\noindent (3) Let $\ell\cong \left(\begin{smallmatrix}a&b\\b&c\end{smallmatrix}\right)$ ($0\le 2b\le a\le c$) be a $\z$-lattice which is primitively represented by neither the genus of $N$ nor the $5$-section of $L$. Then by Lemma~\ref{lem:5excep} and by (1) given above, we may assume that $\ell$ satisfies one of the following conditions:
\begin{itemize}
\item[(i)] $\ell_2\cong \langle 1, -16\rangle$;
\item[(ii)] $\ell_2\cong \langle 1, 64\alpha\rangle$ for some $\alpha\in\z_2$;
\item[(iii)] $\ell_2\cong\langle 4, 16\alpha\rangle$ or $\langle 20, 16\alpha\rangle$ for some $\alpha\in\z_2$;
\item[(iv)] $\fn(\ell_2)\subseteq 16\z_2$.
\end{itemize}
If $c\le 21$, then one may directly check that $\ell$ is primitively represented by $L$. Hence, we may assume that $c\ge 22$. If either
\[
 \ell'(s, t; s') \cong \left(\begin{smallmatrix}a-2s^2 - 8s'^2 & b-2st\\b-2st & c-2t^2\end{smallmatrix}\right)\quad\text{or}\quad \ell''(s, t; t') \cong \left(\begin{smallmatrix}a-2s^2 & b-2st\\b-2st & c-2t^2 - 8t'^2\end{smallmatrix}\right)
\]
satisfies all conditions given in (2), then $\ell$ is primitively represented by $L$.

Assume that case (i) holds. First, suppose that $a$ is odd. Note that $a\equiv 1\Mod8$. Since $d\ell''(0,2;1)\equiv 32\Mod{64}$, $\ell''(0,2;1)_2 \cong \langle 1, 32\epsilon \rangle$ is primitively represented by $N_2$. Since $\ell''(0,2;1)$ is positive definite, $\ell$ is primitively represented by $L$. Now, suppose that $a$ is even. Note that $c\equiv 1\Mod8$, and we have
\[
 a\equiv 20\Mod{32}\text{,} \quad a\equiv -16\Mod{64}\text{,} \quad \text{or} \quad a\equiv 0\Mod{128}\text{.}
\]
Since $\ell'(2,0;1)$ is positive definite, $\ell$ is primitively represented by $L$, by the similar reasoning.

Next, assume that case (ii) holds. Suppose that $a$ is odd. Note that $a\equiv 1\Mod8$. Since $d\ell''(0,0;1)\equiv -8\Mod{64}$, $\ell''(0,0;1)_2 \cong \langle 1, -8 \rangle$ is primitively represented by $N_2$. Furthermore, since $\ell''(0,0;1)$ is positive definite, $\ell$ is primitively represented by $L$. Now, suppose that $a$ is even. Note that $c\equiv 1\Mod8$, and we have
\[
 a\equiv 4\Mod{32}\text{,} \quad a\equiv 16\Mod{64}\text{,} \quad \text{or} \quad a\equiv 0\Mod{64}\text{.}
\]
If $a\ge 16$, then $\ell'(0,0;1)$ is positive definite. Hence, $\ell$ is primitively represented by $L$ in this case. If $a = 4$, then $\ell''(0,0;1)_2 \cong \langle 1, 32\epsilon \rangle$ and $\ell''(0,0;1)$ is positive definite. Hence, $\ell$ is primitively represented by $L$.

Now, assume that case (iii) holds. First, suppose that $a\equiv c\equiv 4\Mod{16}$. Note that $b\equiv 4\Mod8$. One may easily show that there is an $\eta\in\{1, -1\}$ such that $d\ell''(1, \eta; 0)\equiv 32\Mod{64}$. Then, $\ell''(1, \eta; 0)_2 \cong \langle 2, 16\epsilon \rangle$ is primitively represented by $N_2$. Since $\ell''(1, \eta; 0)$ is positive definite, $\ell$ is primitively represented by $L$. Next, suppose that $a\equiv 4\Mod{16}$, $b\equiv 0\Mod8$, and $c\equiv 0\Mod{16}$. In this case, either $\ell''(1,0;0)$ or $\ell''(1,2;0)$ is isometric to $\langle 2, 16\epsilon \rangle$ over $\z_2$. Furthermore, since it is positive definite, $\ell$ is primitively represented by $L$. Similarly to the above, if $a\equiv 0\Mod{16}$, $b\equiv 0\Mod8$, and $c\equiv 4\Mod{16}$, then $\ell$ is primitively represented by $L$.

Finally, assume that case (iv) holds. If we define a $\z$-lattice $\ell'''$ by
\newdimen\ofontdimentwo%
\newdimen\ofontdimensam%
\newdimen\ofontdimennes%
\newdimen\ofontdimenqil%
\ofontdimentwo=\fontdimen2\font%
\ofontdimensam=\fontdimen3\font%
\ofontdimennes=\fontdimen4\font%
\ofontdimenqil=\fontdimen7\font%
\fontdimen2\font=2.7pt%
\fontdimen3\font=1.35pt%
\fontdimen4\font=.9pt%
\fontdimen7\font=.9pt%
\[
 \ell''' = \begin{cases}
  \ell''(0,1;0) & \text{if }a\equiv 16\Mod{32}\text{,}\\
  \ell'(0,1;0) & \text{if }c\equiv 16\Mod{32}\text{,}\\
  \ell'(1,1;0) & \text{if }a\equiv c\equiv 0\Mod{32}\text{ and }b\equiv 8\Mod{16}\text{,}\\
  \ell''(1,0;1) & \text{if }a\equiv 0\Mod{32}\text{, }b\equiv 0\Mod{16}\text{, and }c\equiv 0\Mod{64}\text{,}\\
  \ell'(0,1;1) & \text{if }a\equiv 0\Mod{64}\text{, }b\equiv 0\Mod{16}\text{, and }c\equiv 0\Mod{32}\text{,}\\
  \ell'(1,1;1) & \text{if }a\equiv c\equiv 16\Mod{32}\text{ and }b\equiv 0\Mod{32}\text{,}\\
  \ell'(1,0;0) & \text{if }a\equiv 32\Mod{64}\text{, }b\equiv 16\Mod{32}\text{, }c\equiv -32\Mod{128}\text{,}\\
  \ell'(0,1;0) & \text{if }a\equiv -32\Mod{128}\text{, }b\equiv 16\Mod{32}\text{, }c\equiv 32\Mod{64}\text{,}\\
  \ell'(1,1;0) & \text{if }a\equiv c\equiv 32\Mod{128}\text{ and }b\equiv -16\Mod{64}\text{,}\\
  \ell'(1,3;0) & \text{if }a\equiv c\equiv 32\Mod{128}\text{ and }b\equiv 16\Mod{64}\text{,}
 \end{cases}
\]
\fontdimen2\font=\ofontdimentwo%
\fontdimen3\font=\ofontdimensam%
\fontdimen4\font=\ofontdimennes%
\fontdimen7\font=\ofontdimenqil%
then one may easily show that
\[
 d\ell'''\equiv 16\Mod{128}\text{,}\quad d\ell'''\equiv 32\Mod{64}\text{,}\quad \text{or}\quad d\ell'''\equiv 64\Mod{256}\text{.}
\]
Hence, $\ell$ is primitively represented by $L$ in each case.
\end{proof}
%
%
%
%
%
%

Finally, assume that $L\cong$ I$^{\mbox{\scriptsize ii}}_4$ or I$^{\mbox{\scriptsize iii}}_4$. Note that the quaternary orthogonal summand $\z[e_3, e_4, e_5, e_6]$ of $L$ has a nonsquare discriminant. Hence, such a quaternary $\z$-lattice is not primitively $2$-universal over $\z_p$ for infinitely many primes $p$.

\begin{thm}\label{thm:pfi24}
\noindent\textup{(1)} The $\z$-lattice $N\cong \left(\begin{smallmatrix}2&1&1&1\\1&2&1&1\\1&1&2&0\\1&1&0&4\end{smallmatrix}\right)$ primitively represents any binary $\z$-lattice $\ell'$ satisfying all of the following three conditions:
\begin{itemize}
\item[(a)] $\fn (\ell'_2)\subseteq 2\z_2$ and $\ell'_2$ represents some element in $\{6, -2, 4, 20\}$;
\item[(b)] $\ell'_3$ represents $\Delta_3$ or $3$;
\item[(c)] $\ell'_p$ represents a unit in $\z_p$ for any odd prime $p$ with $\bigl(\frac3p\bigr) = -1$, where $\bigl(\frac{\cdot}{\cdot}\bigr)$ is the Legendre symbol.
\end{itemize}
\noindent\textup{(2)} If the quinary $\z$-lattice $M\cong \ang{2}\mathbin{\perp} N$ does not primitively represent a positive definite binary $\z$-lattice $\ell$, then $\ell$ satisfies $\fn (\ell_2) = \z_2$, $\ell_3\cong \ang{3\cdot\Delta_3, 9\alpha}$ for some $\alpha\in\z_3$, or $\fs (\ell_3)\subseteq 9\z_3$.

\noindent\textup{(3)} The $\z$-lattice \textup{I$^{\mbox{\scriptsize ii}}_4$} $\cong I_2\mathbin{\perp} N$ is primitively $2$-universal.
\end{thm}

\begin{proof}
(1) Note that $N$ has class number one and $dN = 12$. Hence, any binary $\z$-lattice $\ell'$ is primitively represented by $N$ if and only if $\ell'_p$ is primitively represented by $N_p$ for any prime $p$. Since $N_2\cong \mathbb{A}\mathbin{\perp}\ang{-2, -2}\cong \mathbb{H}\mathbin{\perp}\ang{6, -2}$, $\ell'_2$ is primitively represented by $N_2$ if $\fn (\ell'_2)\subseteq(2)$ and $\ell'_2$ represents some element in $\{6, -2, 4, 20\}\subset\z_2$. Since $N_3\cong \mathbb{H}\mathbin{\perp}\ang{\Delta_3, 3}$, $\ell'_3$ is primitively represented by $N_3$ if $\ell'_3$ represents $\Delta_3$ or $3$. Now, suppose $p\ne 2$, $3$. If $\bigl(\frac3p\bigr) = 1$ then $N_p\cong \mathbb{H}\mathbin{\perp} \mathbb{H}$, and hence $N_p$ is primitively $2$-universal. If $\bigl(\frac3p\bigr) = -1$, that is, $dN_p = \Delta_p$, then $N_p\cong \mathbb{H}\mathbin{\perp}\ang{1, -\Delta_p}$. Hence, $\ell'_p$ is primitively represented by $N_p$ if it represents a unit in $\z_p$. Note that for any odd prime $p$, $\bigl(\frac3p\bigr) = -1$ if and only if $p\equiv 5$, $7\Mod{12}$.

\noindent (2) Note that $M$ has class number one, $dM = 24$, and $3$ is the only core prime of $M$. Hence, any binary lattice $\ell$ is primitively represented by $M$ if and only if $\ell_p$ is primitively represented by $M_p$ for $p=2$, $3$. Note that $M_2\cong \mathbb{H}\mathbin{\perp} \mathbb{H}^2\mathbin{\perp} \ang{6}$ primitively represents any binary lattice $\ell_2$ satisfying $\fn (\ell_2)\subseteq 2\z_2$, and $M_3\cong \mathbb{H}\mathbin{\perp} \ang{1,1,3}$ primitively represents all binary lattices representing $1$, $\Delta_3$, or $3$.

\noindent (3) Let $\ell\cong \left(\begin{smallmatrix}a&b\\b&c\end{smallmatrix}\right)$ ($0\le 2b\le a\le c$) be a binary $\z$-lattice which is primitively represented by neither the $5$-section of $L$ nor the primitive sublattice
\[
 M = \z[e_1+e_2, e_3, e_4, e_5, e_6]\cong \ang{2}\mathbin{\perp} N
\]
of $L$. Then by Lemma~\ref{lem:5excep} and by (1) given above, we may assume that $\ell$ satisfies one of the following conditions:
\begin{itemize}
\item[(i)] $\ell_2\cong \langle 1, 32\alpha\rangle$ for some $\alpha\in\z_2$;
\item[(ii)] $\ell_2\cong \langle 5, 16\epsilon\rangle$ for some $\epsilon\in\z_2^\times$;
\item[(iii)] $a\equiv b\equiv c\equiv 0\Mod{12}$.
\end{itemize}

First, assume that case (iii) holds. Observe that $M^\perp = \z[e_1 - e_2] \cong \ang{2}$ and $e_1 - e_2 = - (e_1 + e_2) + 2e_1$. Hence, if $\ell'\cong \left(\begin{smallmatrix}a&b\\b&c-2\cdot 2^2\end{smallmatrix}\right)$ is primitively represented by $M$, then $\ell$ is primitively represented by $L$ by Lemma~\ref{lem:notog}. In fact, $\ell'$ is primitively represented by $M$ for $\fs (\ell'_2)\subseteq 4\z_2$ and $\fs (\ell'_3) = \z_3$.

Denote by $O$ the $5$-section of $L$. Then $O^\perp = \z(- e_3 - e_4 + e_5 + 2e_6)\cong \ang{12}$. Hence, if
\[
 \ell''\cong \begin{pmatrix}a-12\cdot 2^2&b\\b&c\end{pmatrix}\quad \text{or}\quad \ell'''\cong \begin{pmatrix}a&b\\b&c-12\cdot 2^2\end{pmatrix}
\]
is primitively represented by $O$, then $\ell$ is primitively represented by $L$ by Lemma~\ref{lem:notog}. Moreover, $\ell''$ ($\ell'''$) is primitively represented by $O$ if and only if $\ell''$ ($\ell'''$) is positive definite and $\ell''_2$ ($\ell'''_2$, respectively) is primitively represented by $O_2$.

Now, assume that case (i) holds. If $c\le 64$, then one may directly check that $\ell$ is primitively represented by $L$. Now, we assume that $c\ge 65$. First, suppose that $a$ is odd. Since $d\ell'''\equiv 16\Mod{32}$, $\ell'''_2$ is primitively represented by $O_2$. Furthermore, since $\ell'''$ is positive definite, it is primitively represented by $O$. Now, suppose that $a$ is even. Since $c$ is odd, similarly to the above, $\ell''$ is primitively represented by $O$ if $a\ge 64$ so that $\ell''$ is positive definite. Hence, we may assume that $a < 64$. Note that $c\equiv 1\Mod8$ and one of the following conditions holds: 
\begin{itemize}
\item[(I)] $a\equiv 4\Mod{32}$ and $b\equiv 2\Mod4$;
\item[(II)] $a\equiv 16\Mod{32}$ and $b\equiv 4\Mod8$;
\item[(III)] $a\equiv 0\Mod{32}$ and $b\equiv 0\Mod8$.
\end{itemize}
We define
\[
 \ell^{(4)} \cong \begin{cases}
  \left(\begin{smallmatrix}a&b\\b&c-1-4\end{smallmatrix}\right) & \text{if }a=48\text{,}\\
  \left(\begin{smallmatrix}a-4&b\\b&c-1\end{smallmatrix}\right) & \text{if }a=36\text{,}\\
  \left(\begin{smallmatrix}a&b\\b&c-9-4\end{smallmatrix}\right) & \text{if }a=32\text{ and }c\equiv 1\Mod{16}\text{,}\\
  \left(\begin{smallmatrix}a&b\\b&c-1-4\end{smallmatrix}\right) & \text{if }a=32\text{ and }c\equiv 9\Mod{16}\text{,}\\
  \left(\begin{smallmatrix}a&b\\b&c-9-4\end{smallmatrix}\right) & \text{if }a=16\text{ and }c\equiv 0\Mod3\text{,}\\
  \left(\begin{smallmatrix}a&b\\b&c-1-4\end{smallmatrix}\right) & \text{if }a=16\text{ and }c\not\equiv 0\Mod3\text{,}\\
  \left(\begin{smallmatrix}a&b\\b&c-1\end{smallmatrix}\right) & \text{if }a=4\text{ and }c\equiv 0,1\Mod3\text{,}\\
  \left(\begin{smallmatrix}a&b\\b&c-9\end{smallmatrix}\right) & \text{if }a=4\text{ and }c\equiv 2\Mod3\text{.}
 \end{cases}
\]
Then by (1), $\ell^{(4)}$ is primitively represented by $N$. Hence, $\ell$ is primitively represented by $L$ in each case. The proof of case (ii) is quite similar to this.
\end{proof}

\begin{thm}
\noindent\textup{(1)} The $\z$-lattice $N\cong \left(\begin{smallmatrix}2&1&1&1\\1&2&1&1\\1&1&2&1\\1&1&1&4\end{smallmatrix}\right)$ primitively represents a positive definite binary $\z$-lattice $\ell'$ if all of the following three conditions hold:
\begin{itemize}
\item[(a)] $\fn\ell'_2\subseteq 2\z_2$ and $\ell'_2$ represents twice an odd integer;
\item[(b)] $\ell'_{13}$ represents $1$ or $13$;
\item[(c)] $\ell'_p$ represents a unit in $\z_p$ for any odd prime $p$ with $\bigl(\frac{13}{p}\bigr) = \bigl(\frac{p}{13}\bigr) = -1$, where $\bigl(\frac{\cdot}{\cdot}\bigr)$ is the Legendre symbol.
\end{itemize}
\noindent\textup{(2)} The lattice \textup{I$^{\mbox{\scriptsize iii}}_4$} $\cong I_2\mathbin{\perp} N$ is primitively $2$-universal.
\end{thm}

\begin{proof}
Since the proof is quite similar to that of the above theorem, it is left to the readers.
\end{proof}

\subsection{Case 6: The remaining cases}
\label{ssec:55}

In this subsection, we prove the primitive $2$-universalities of the remaining lattices
\[
 \text{D$^{\mbox{\scriptsize ii}}_k$ ($4\le k\le 8$), D$^{\mbox{\scriptsize iii}}_6$, D$^{\mbox{\scriptsize iii}}_7$, H$^{\mbox{\scriptsize iii}}_6$, and I$^{\mbox{\scriptsize iii}}_2$.}
\]
Note that the $5$-section of $L$ does not split $L$ orthogonally in all cases.

Assume that $L\cong$ D$^{\mbox{\scriptsize ii}}_k$. Note that one may prove the primitive $2$-universality of $L$ in the similar manner to Theorem~\ref{thm:pf43} by using Lemma~\ref{lem:superd2h} instead of Lemma~\ref{lem:357super}.

Next, assume that $L\cong$ D$^{\mbox{\scriptsize iii}}_k$ for $k = 6$ or $7$.

\begin{lem}\label{lem:superd3}
Let $\ell \cong \left(\begin{smallmatrix}a&b\\b&c\end{smallmatrix}\right)$ be a $\z$-lattice and let $m$, $k\ge 2$ be positive integers. Suppose that there exists a primitive sublattice
\[
 \z\begin{bmatrix}c_1 & \cdots & c_m & c_{m+1} & c_{m+2}\\d_1 & \cdots & d_m & d_{m+1} & d_{m+2}\end{bmatrix} \cong \begin{pmatrix}a-(k-1)s^2 & b-(k-1)st\\b-(k-1)st & c-(k-1)t^2\end{pmatrix}
\]
of the $\z$-lattice $I_{m+2}$ for some integers $c_i$, $d_i$, $s$, and $t$. If
\[
 c_{m+1} + c_{m+2} + s \equiv d_{m+1} + d_{m+2} + t \equiv 0\Mod2\text{,}
\]
then $\ell$ is primitively represented by $I_m\mathbin{\perp}\left(\begin{smallmatrix}2&0&1\\0&2&1\\1&1&k\end{smallmatrix}\right)$.
\end{lem}

\begin{proof}
Since the proof is quite similar to that of Lemma~\ref{lem:357super}, it is left to the readers.
\end{proof}

\begin{lem}\label{lem:d3coeff}
Let $\ell\cong \left(\begin{smallmatrix}a&b\\b&c\end{smallmatrix}\right)$ be a binary $\z$-lattice. Suppose that there exists a primitive sublattice $\z\left[\begin{smallmatrix}c_1 & c_2 & c_3 & c_4 & c_5\\d_1 & d_2 & d_3 & d_4 & d_5\end{smallmatrix}\right]$ of the $\z$-lattice $I_5$, which is isometric to $\ell$, for some integers $c_i$ and $d_i$ such that the set
\[
 \{(\overline{\vphantom{d}c_i + c_j}, \overline{d_i + d_j})\mid 1\le i < j\le 5\}
\]
is a proper subset of $(\z/2\z)^2$, where $\overline{n}$ is the residue class of $n$ modulo $2$ for any integer $n$. Then $\ell$ satisfies one of the followings:
\begin{itemize}
\item[\textup{(i)}] $a$ or $c\equiv 1\Mod4$ and $d\ell\equiv 0\Mod4$;
\item[\textup{(ii)}] $a\not\equiv 1\Mod8$, $c\not\equiv 1\Mod8$, and $d\ell\equiv 2\Mod4$.
\end{itemize}
\end{lem}

\begin{proof}
Consider the set $C = \{(\overline{\vphantom{d}c_i}, \overline{d_i})\mid 1\le i \le 5\}$. Since $\z\left[\begin{smallmatrix}c_1 & c_2 & c_3 & c_4 & c_5\\d_1 & d_2 & d_3 & d_4 & d_5\end{smallmatrix}\right]$ is a primitive sublattice of $I_5$, the set $C$ contains at least two nonzero vectors in $(\z/2\z)^2$. Furthermore, one may easily show from the assumption that $C$ is one of 
\[
 \{(\overline{1}, \overline{0}), (\overline{0}, \overline{1})\}\text{,}\quad \{(\overline{1}, \overline{0}), (\overline{1}, \overline{1})\}\text{,}\quad \{(\overline{0}, \overline{1}), (\overline{1}, \overline{1})\}\text{,}
\]
which corresponds to each of the followings, respectively:
\begin{itemize}
\item[\textup{(a)}] $a+c\equiv 1\Mod4$ and $b$ is even;
\item[\textup{(b)}] $a\equiv 5\Mod8$ and $b\equiv c\Mod2$;
\item[\textup{(c)}] $a\equiv b\Mod2$ and $c\equiv 5\Mod8$.
\end{itemize}
The lemma follows directly from this.
\end{proof}

\begin{thm}\label{thm:pfd3}
The $\z$-lattice \textup{D$^{\mbox{\scriptsize iii}}_k$} $\cong I_3\mathbin{\perp}\left(\begin{smallmatrix}2&0&1\\0&2&1\\1&1&k\end{smallmatrix}\right)$ for $k = 6$ or $7$ is primitively $2$-universal.
\end{thm}

\begin{proof}
Let $\ell\cong \left(\begin{smallmatrix}a&b\\b&c\end{smallmatrix}\right)$ be a binary $\z$-lattice such that $0\le 2b\le a\le c$. Note that the $5$-section $M \cong I_3 \mathbin{\perp} \langle 2, 2 \rangle$ in this case has class number one. Hence, we may assume that $\ell$ is not primitively represented by $M$ locally, that is, one of the following conditions holds:
\begin{itemize}
\item[(i)] $\ell_2\cong\langle 1, 16\alpha\rangle$ for some $\alpha\in\z_2$;
\item[(ii)] $\ell_2\cong\langle 4, 16\alpha\rangle$ or $\langle 20, 16\alpha\rangle$ for some $\alpha\in\z_2$;
\item[(iii)] $\fn(\ell_2)\subseteq 16\z_2$.
\end{itemize}
Note that we have $a\equiv 1\Mod8$, $a\equiv 4\Mod{16}$, or $a\equiv 0\Mod{16}$. Assume that both of the $\z$-lattices
\[
 \ell(1) \cong \begin{pmatrix}a-(k-1)&b\\b&c\end{pmatrix}\text{,}\qquad \ell(2) \cong \begin{pmatrix}a&b\\b&c-(k-1)\end{pmatrix}
\]
are positive definite. Then one may easily show that $\ell(s)$ is primitively represented by $I_5$ for some $s = 1$, $2$. Let $N = \z\left[\begin{smallmatrix}c_1 & c_2 & c_3 & c_4 & c_5\\d_1 & d_2 & d_3 & d_4 & d_5\end{smallmatrix}\right]$ be a primitive binary $\z$-sublattice of $I_5$ which is isometric to $\ell(s)$. If there is an $(i, j)$ with $1\le i<j\le 5$ such that
\[
 c_i + c_j + (2-s) \equiv d_i + d_j + (s-1)\equiv 0\Mod2\text{,}
\]
then $\ell$ is primitively represented by \textup{D$^{\mbox{\scriptsize iii}}_k$} by Lemma~\ref{lem:superd3}. If there does not exist such an $(i, j)$, then by Lemma~\ref{lem:d3coeff}, one of the followings must hold:
\begin{itemize}
\item[(a)] $a-(2-s)(k-1)\equiv 1\Mod4$ or $c-(s-1)(k-1)\equiv 1\Mod4$, and $d\ell(s)\equiv 0\Mod4$;
\item[(b)] $a-(2-s)(k-1)\not\equiv 1\Mod8$, $c-(s-1)(k-1)\not\equiv 1\Mod8$, and $d\ell(s)\equiv 2\Mod4$.
\end{itemize}
However, one may easily verify that none of (a) and (b) holds in each case. For instance, consider case (i) when $k = 6$. If $a$ is odd, then $a\equiv 1\Mod8$. Since $d\ell(2)\equiv 3\Mod8$, $\ell(2)$ is primitively represented by $I_5$ so that we may take $s = 2$. Since $d\ell(2)$ is odd, $\ell(2)$ satisfies neither (a) nor (b). Now, suppose that $a$ is even. Then $a\equiv 0\Mod4$ and $c\equiv 1\Mod8$. Therefore, similarly to the above, $\ell(1)$ is primitively represented by $I_5$ and hence $\ell(1)$ does not satisfy any of (a) and (b).

Now, we have to consider the case when neither $\ell(1)$ nor $\ell(2)$ is positive definite. Note that if $a\ge 9$, then both $\ell(1)$ and $\ell(2)$ are positive definite. Hence, we may assume that $a = 1$ or $a = 4$. If $a = 1$, then $b = 0$ and $c\equiv 0\Mod8$ by (i). Since $\ell(2)$ is positive definite if $c\ge 9$, one may apply the same method as the above to prove the theorem. One may directly check that $\ell$ is primitively represented by $L$ if $c\le 8$. Now, assume that $a = 4$. Then we have either $b = 0$ or $b = 2$. If $b = 0$, then $c\equiv 0\Mod{16}$ by (ii). Hence, $\ell(2)$ is positive definite, and we may apply the same method to the above to prove the theorem. If $b = 2$, then $c\equiv 1\Mod8$ by (i). Note that $I_3$ represents $c-k\equiv 2, 3\Mod8$ by Legendre's three-square theorem. If we choose a vector $v$ in the $3$-section of $L$ such that $Q(v) = c-k$, then clearly, $\z[e_4 + e_5, v + e_6]$ is a primitive sublattice of $L$ isometric to $\left(\begin{smallmatrix}4&2\\2&c\end{smallmatrix}\right)$.
\end{proof}

Now, assume $L\cong$ H$^{\mbox{\scriptsize iii}}_6$ or I$^{\mbox{\scriptsize iii}}_2$. We may apply Lemma~\ref{lem:notog} by putting $M = \z[e_1 + e_2, e_3, e_4, e_5, e_6]$ and $n=5$. One may prove the primitive $2$-universality of $L$ for these cases in a similar manner to the proofs of Theorems~\ref{thm:pfj} or \ref{thm:pfd35} by using Lemmas~\ref{lem:h} and \ref{lem:notog}.

Finally, we provide some essential data needed for computations in this section.

\begin{lem}\label{lem:4excep}
For each given quaternary $\z_2$-lattice $N$, the binary $\z_2$-lattice $\ell$ that is not primitively represented by $N$ satisfies one of the conditions given in Table~\ref{tblr:4excep}.
\end{lem}

\begin{proof}
Since one may prove the lemma by direct computations, the proof is left to the readers.
\end{proof}
\begin{table}[ht]
\caption{The local structures}
\label{tblr:4excep}
\renewcommand{\arraystretch}{1.2}
\newdimen\ofontdimentwo%
\newdimen\ofontdimensam%
\newdimen\ofontdimennes%
\newdimen\ofontdimenqil%
\ofontdimentwo=\fontdimen2\font%
\ofontdimensam=\fontdimen3\font%
\ofontdimennes=\fontdimen4\font%
\ofontdimenqil=\fontdimen7\font%
\fontdimen2\font=2.25pt%
\fontdimen3\font=1.125pt%
\fontdimen4\font=.75pt%
\fontdimen7\font=.75pt%
\begin{tabular}{c|l}
 \hline
 $N$ & \hfil Binary $\z_2$-lattices that are not primitively represented by $N$ \\\hline
 $I_4$ & $\ell_2\cong\ang{\epsilon, 4\alpha}$, $\ang{2, 6}$, $\ang{2\epsilon, 8\alpha}$, $\fn (\ell_2)\subseteq 4\z_2$, or $\q_2 \ell\cong \q_2 \mathbb{H}$\\\hline
 $\ang{1,1,2,2}$ &
 \begin{tabular}{@{}l@{}}
  $\ell_2\cong\ang{\epsilon, 4\delta}$ with $\epsilon\delta\equiv 3\Mod8$, $\ang{\epsilon, 16\alpha}$, $\ang{2\epsilon, 8\alpha}$,\\[-2pt]
  $\ang{4\epsilon, 4\delta}$ with $\epsilon\delta\equiv 3\Mod8$, $\ang{4\epsilon, 16\alpha}$, $\mathbb{A}$, $\fn (\ell_2)\subseteq 8\z_2$,\\[-2pt]
  or $\q_2 \ell\cong \q_2 \mathbb{H}$
 \end{tabular}\\\hline
 $\ang{1,2,2,4}$ &
 \begin{tabular}{@{}l@{}}
  $\ell_2$ is unimodular, $\ell_2\cong\ang{\epsilon, 16\alpha}$, $\ang{2\epsilon, 8\delta}$ with $\epsilon\delta\equiv 3\Mod8$,\\[-2pt]
  $\ang{2\epsilon, 32\alpha}$, $\ang{4\epsilon, 16\alpha}$, $\ang{8\epsilon, 8\delta}$ with $\epsilon\delta\equiv 3\Mod8$, $\ang{8\epsilon, 32\alpha}$,\\[-2pt]
  $\mathbb{A}^2$, $\fn (\ell_2)\subseteq 16\z_2$, or $\q_2 \ell\cong \q_2 \mathbb{H}$
 \end{tabular}\\\hline
 $\ang{1,2}\mathbin{\perp}\left(\begin{smallmatrix}3&1\\1&3\end{smallmatrix}\right)$ &
 \begin{tabular}{@{}l@{}}
  $\ell_2\cong\ang{1, 1}$, $\ang{3, -1}$, $\ang{1, 20}$, $\ang{-1, -4}$,\\[-2pt]
  $\ang{\epsilon, 16\delta}$ with $\epsilon\delta\equiv 3\Mod8$, $\ang{\epsilon, 64\alpha}$, $\ang{2, 2}$, $\ang{2, 6}$, $\ang{2, 10}$,\\[-2pt]
  $\ang{2\epsilon, 16\alpha}$ with $\epsilon\equiv 1\Mod4$, $\ang{2\epsilon, 32\alpha}$ with $\epsilon\equiv -1\Mod4$,\\[-2pt]
  $\ang{12, 12}$, $\ang{4, 20}$, $\ang{4\epsilon, 16\delta}$ with $\epsilon\delta\equiv 3\Mod8$, $\ang{4\epsilon, 64\alpha}$,\\[-2pt]
  $\fn (\ell_2)\subseteq 8\z_2$, or $\q_2 \ell \cong \q_2 \mathbb{H}$
 \end{tabular}\\\hline
\end{tabular}
\fontdimen2\font=\ofontdimentwo%
\fontdimen3\font=\ofontdimensam%
\fontdimen4\font=\ofontdimennes%
\fontdimen7\font=\ofontdimenqil%
\end{table}

\begin{lem}\label{lem:5excfj}
For the $5$-section $M$ and its core prime $q$ of each type given in Table~\ref{tblr:5excfj}, if a binary $\z$-lattice $\ell$ is not primitively represented by $M$, then $\ell$ satisfies one of the conditions given in the table.
\end{lem}

\begin{proof}
Since one may prove the lemma by direct computations, the proof is left to the readers.
\end{proof}
\begin{table}[ht]
\caption{The core prime and the local structures}
\label{tblr:5excfj}
\renewcommand{\arraystretch}{1.2}
\begin{tabular}{c|c|c|l}
 \hline
 Type & $M$ & $q$ & \hfil Local structures \\\hline
 C & $I_4\mathbin{\perp}\ang{3}$ & $2$ & $\ell_2\cong\ang{3, 8\alpha}$ or $\fn (\ell_2)\subseteq 4\z_2$\\\hline
 F & $I_3\mathbin{\perp}\ang{2, 3}$ & $3$ &
  \begin{tabular}{@{}l@{}}
   $\ell_2\cong\ang{4\epsilon, 4\delta}$ or $\mathbb{A}$, or\\[-2pt]
   $\ell_3\cong\ang{6, 9\alpha}$ or $\fs (\ell_3)\subseteq 9\z_3$
  \end{tabular}\\\hline
 J & $I_2\mathbin{\perp}\left(\begin{smallmatrix}2&1\\1&2\end{smallmatrix}\right)\mathbin{\perp}\ang{3}$ & $2$ & $\ell_2\cong\ang{1, 8\alpha}$ or $\fn (\ell_2)\subseteq 4\z_2$\\\hline
 - & $I_3\mathbin{\perp}\left(\begin{smallmatrix}3&1\\1&3\end{smallmatrix}\right)$ & $2$ &
  \begin{tabular}{@{}l@{}}
   $\ell_2\cong\ang{3, 7}$, $\ang{-1, 4}$, $\ang{2, 2}$, $\ang{2, 64\alpha}$,\\[-2pt]
   $\ang{10, 32\epsilon}$, $\ang{8, 64\alpha}$, or $\fs (\ell_2)\subseteq 16\z_2$
  \end{tabular}\\\hline
\end{tabular}
\end{table}

\end{document}